\documentclass[a4paper, british]{amsart}

\usepackage{libertine}
\usepackage[libertine]{newtxmath}

\usepackage{babel}
\usepackage{enumitem}
\usepackage{hyperref}
\usepackage[utf8]{inputenc}
\usepackage{newunicodechar}
\usepackage{mathtools}
\usepackage{varioref}
\usepackage[arrow,curve,matrix]{xy}

\usepackage{colortbl}
\usepackage{graphicx}
\usepackage{tikz}

\usepackage{mathrsfs}

\definecolor{linkred}{rgb}{0.7,0.2,0.2}
\definecolor{linkblue}{rgb}{0,0.2,0.6}

\numberwithin{figure}{section}

\usepackage[hyperpageref]{backref}

\setdescription{labelindent=\parindent, leftmargin=2\parindent}
\setitemize[1]{labelindent=\parindent, leftmargin=2\parindent}
\setenumerate[1]{labelindent=0cm, leftmargin=*, widest=iiii}

\newunicodechar{א}{\ensuremath{\aleph}}
\newunicodechar{α}{\ensuremath{\alpha}}
\newunicodechar{β}{\ensuremath{\beta}}
\newunicodechar{χ}{\ensuremath{\chi}}
\newunicodechar{δ}{\ensuremath{\delta}}
\newunicodechar{ε}{\ensuremath{\varepsilon}}
\newunicodechar{Δ}{\ensuremath{\Delta}}
\newunicodechar{η}{\ensuremath{\eta}}
\newunicodechar{γ}{\ensuremath{\gamma}}
\newunicodechar{Γ}{\ensuremath{\Gamma}}
\newunicodechar{ι}{\ensuremath{\iota}}
\newunicodechar{κ}{\ensuremath{\kappa}}
\newunicodechar{λ}{\ensuremath{\lambda}}
\newunicodechar{Λ}{\ensuremath{\Lambda}}
\newunicodechar{ν}{\ensuremath{\nu}}
\newunicodechar{μ}{\ensuremath{\mu}}
\newunicodechar{ω}{\ensuremath{\omega}}
\newunicodechar{Ω}{\ensuremath{\Omega}}
\newunicodechar{π}{\ensuremath{\pi}}
\newunicodechar{Π}{\ensuremath{\Pi}}
\newunicodechar{φ}{\ensuremath{\phi}}
\newunicodechar{Φ}{\ensuremath{\Phi}}
\newunicodechar{ψ}{\ensuremath{\psi}}
\newunicodechar{Ψ}{\ensuremath{\Psi}}
\newunicodechar{ρ}{\ensuremath{\rho}}
\newunicodechar{σ}{\ensuremath{\sigma}}
\newunicodechar{Σ}{\ensuremath{\Sigma}}
\newunicodechar{τ}{\ensuremath{\tau}}
\newunicodechar{θ}{\ensuremath{\theta}}
\newunicodechar{Θ}{\ensuremath{\Theta}}
\newunicodechar{ξ}{\ensuremath{\xi}}
\newunicodechar{Ξ}{\ensuremath{\Xi}}
\newunicodechar{ζ}{\ensuremath{\zeta}}

\newunicodechar{ℓ}{\ensuremath{\ell}}
\newunicodechar{ï}{\"{\i}}

\newunicodechar{𝔸}{\ensuremath{\bA}}
\newunicodechar{𝔹}{\ensuremath{\bB}}
\newunicodechar{ℂ}{\ensuremath{\bC}}
\newunicodechar{𝔻}{\ensuremath{\bD}}
\newunicodechar{𝔼}{\ensuremath{\bE}}
\newunicodechar{𝔽}{\ensuremath{\bF}}
\newunicodechar{𝔾}{\ensuremath{\bG}}
\newunicodechar{ℕ}{\ensuremath{\bN}}
\newunicodechar{ℙ}{\ensuremath{\bP}}
\newunicodechar{ℚ}{\ensuremath{\bQ}}
\newunicodechar{ℝ}{\ensuremath{\bR}}
\newunicodechar{𝕏}{\ensuremath{\bX}}
\newunicodechar{ℤ}{\ensuremath{\bZ}}
\newunicodechar{𝒜}{\ensuremath{\sA}}
\newunicodechar{ℬ}{\ensuremath{\sB}}
\newunicodechar{𝒞}{\ensuremath{\sC}}
\newunicodechar{𝒟}{\ensuremath{\sD}}
\newunicodechar{ℰ}{\ensuremath{\sE}}
\newunicodechar{ℱ}{\ensuremath{\sF}}
\newunicodechar{𝒢}{\ensuremath{\sG}}
\newunicodechar{ℋ}{\ensuremath{\sH}}
\newunicodechar{𝒥}{\ensuremath{\sJ}}
\newunicodechar{ℒ}{\ensuremath{\sL}}
\newunicodechar{𝒪}{\ensuremath{\sO}}
\newunicodechar{𝒬}{\ensuremath{\sQ}}
\newunicodechar{𝒮}{\ensuremath{\sS}}
\newunicodechar{𝒯}{\ensuremath{\sT}}
\newunicodechar{𝒲}{\ensuremath{\sW}}

\newunicodechar{∂}{\ensuremath{\partial}}
\newunicodechar{∇}{\ensuremath{\nabla}}

\newunicodechar{↺}{\ensuremath{\circlearrowleft}}
\newunicodechar{∞}{\ensuremath{\infty}}
\newunicodechar{⊕}{\ensuremath{\oplus}}
\newunicodechar{⊗}{\ensuremath{\otimes}}
\newunicodechar{•}{\ensuremath{\bullet}}
\newunicodechar{Λ}{\ensuremath{\wedge}}
\newunicodechar{↪}{\ensuremath{\into}}
\newunicodechar{→}{\ensuremath{\to}}
\newunicodechar{↦}{\ensuremath{\mapsto}}
\newunicodechar{⨯}{\ensuremath{\times}}
\newunicodechar{∪}{\ensuremath{\cup}}
\newunicodechar{∩}{\ensuremath{\cap}}
\newunicodechar{⊋}{\ensuremath{\supsetneq}}
\newunicodechar{⊇}{\ensuremath{\supseteq}}
\newunicodechar{⊃}{\ensuremath{\supset}}
\newunicodechar{⊊}{\ensuremath{\subsetneq}}
\newunicodechar{⊆}{\ensuremath{\subseteq}}
\newunicodechar{⊂}{\ensuremath{\subset}}
\newunicodechar{⊄}{\ensuremath{\not \subset}}
\newunicodechar{≥}{\ensuremath{\geq}}
\newunicodechar{≠}{\ensuremath{\neq}}
\newunicodechar{≫}{\ensuremath{\gg}}
\newunicodechar{≪}{\ensuremath{\ll}}

\newunicodechar{≤}{\ensuremath{\leq}}
\newunicodechar{∈}{\ensuremath{\in}}
\newunicodechar{∉}{\ensuremath{\not \in}}
\newunicodechar{∖}{\ensuremath{\setminus}}
\newunicodechar{◦}{\ensuremath{\circ}}
\newunicodechar{°}{\ensuremath{^\circ}}
\newunicodechar{…}{\ifmmode\mathellipsis\else\textellipsis\fi}
\newunicodechar{·}{\ensuremath{\cdot}}
\newunicodechar{⋯}{\ensuremath{\cdots}}
\newunicodechar{∅}{\ensuremath{\emptyset}}
\newunicodechar{⇒}{\ensuremath{\Rightarrow}}

\newunicodechar{⁰}{\ensuremath{^0}}
\newunicodechar{¹}{\ensuremath{^1}}
\newunicodechar{²}{\ensuremath{^2}}
\newunicodechar{³}{\ensuremath{^3}}
\newunicodechar{⁴}{\ensuremath{^4}}
\newunicodechar{⁵}{\ensuremath{^5}}
\newunicodechar{⁶}{\ensuremath{^6}}
\newunicodechar{⁷}{\ensuremath{^7}}
\newunicodechar{⁸}{\ensuremath{^8}}
\newunicodechar{⁹}{\ensuremath{^9}}
\newunicodechar{ⁱ}{\ensuremath{^i}}

\newunicodechar{⌈}{\ensuremath{\lceil}}
\newunicodechar{⌉}{\ensuremath{\rceil}}
\newunicodechar{⌊}{\ensuremath{\lfloor}}
\newunicodechar{⌋}{\ensuremath{\rfloor}}

\newunicodechar{≅}{\ensuremath{\cong}}
\newunicodechar{⇔}{\ensuremath{\Leftrightarrow}}
\newunicodechar{∃}{\ensuremath{\exists}}
\newunicodechar{±}{\ensuremath{\pm}}

\DeclareFontFamily{OMS}{rsfs}{\skewchar\font'60}
\DeclareFontShape{OMS}{rsfs}{m}{n}{<-5>rsfs5 <5-7>rsfs7 <7->rsfs10 }{}
\DeclareSymbolFont{rsfs}{OMS}{rsfs}{m}{n}
\DeclareSymbolFontAlphabet{\scr}{rsfs}
\DeclareSymbolFontAlphabet{\scr}{rsfs}

\DeclareFontFamily{U}{mathx}{\hyphenchar\font45}
\DeclareFontShape{U}{mathx}{m}{n}{
      <5> <6> <7> <8> <9> <10>
      <10.95> <12> <14.4> <17.28> <20.74> <24.88>
      mathx10
      }{}
\DeclareSymbolFont{mathx}{U}{mathx}{m}{n}
\DeclareFontSubstitution{U}{mathx}{m}{n}
\DeclareMathAccent{\wcheck}{0}{mathx}{"71}

\DeclareMathOperator{\Aut}{Aut}
\DeclareMathOperator{\codim}{codim}

\DeclareMathOperator{\Ext}{Ext}

\DeclareMathOperator{\Id}{Id}
\DeclareMathOperator{\Image}{Image}
\DeclareMathOperator{\img}{img}
\DeclareMathOperator{\Pic}{Pic}
\DeclareMathOperator{\rank}{rank}

\DeclareMathOperator{\reg}{reg}

\DeclareMathOperator{\sEnd}{\sE\negthinspace \mathit{nd}}
\DeclareMathOperator{\sing}{sing}
\DeclareMathOperator{\Spec}{Spec}
\DeclareMathOperator{\Sym}{Sym}

\DeclareMathOperator{\tor}{tor}

\newcommand{\sA}{\scr{A}}
\newcommand{\sB}{\scr{B}}
\newcommand{\sC}{\scr{C}}
\newcommand{\sD}{\scr{D}}
\newcommand{\sE}{\scr{E}}
\newcommand{\sF}{\scr{F}}
\newcommand{\sG}{\scr{G}}
\newcommand{\sH}{\scr{H}}
\newcommand{\sHom}{\scr{H}\negthinspace om}

\newcommand{\sJ}{\scr{J}}

\newcommand{\sL}{\scr{L}}

\newcommand{\sN}{\scr{N}}
\newcommand{\sO}{\scr{O}}

\newcommand{\sQ}{\scr{Q}}

\newcommand{\sS}{\scr{S}}
\newcommand{\sT}{\scr{T}}

\newcommand{\sW}{\scr{W}}

\newcommand{\bA}{\mathbb{A}}
\newcommand{\bB}{\mathbb{B}}
\newcommand{\bC}{\mathbb{C}}
\newcommand{\bD}{\mathbb{D}}
\newcommand{\bE}{\mathbb{E}}
\newcommand{\bF}{\mathbb{F}}
\newcommand{\bG}{\mathbb{G}}

\newcommand{\bN}{\mathbb{N}}

\newcommand{\bP}{\mathbb{P}}
\newcommand{\bQ}{\mathbb{Q}}
\newcommand{\bR}{\mathbb{R}}

\newcommand{\bX}{\mathbb{X}}

\newcommand{\bZ}{\mathbb{Z}}

\theoremstyle{plain}
\newtheorem{thm}{Theorem}[section]

\newtheorem{cor}[thm]{Corollary}
\newtheorem{defn}[thm]{Definition}
\newtheorem{fact}[thm]{Fact}
\newtheorem{lem}[thm]{Lemma}

\newtheorem{problem}[thm]{Problem}
\newtheorem{prop}[thm]{Proposition}

\theoremstyle{remark}

\newtheorem{c-n-d}[thm]{Claim and Definition}

\newtheorem{construction}[thm]{Construction}

\newtheorem{example}[thm]{Example}

\newtheorem{notation}[thm]{Notation}

\newtheorem{rem}[thm]{Remark}

\newtheorem*{rem-nonumber}{Remark}

\numberwithin{equation}{thm}

\setlist[enumerate]{label=(\thethm.\arabic*), before={\setcounter{enumi}{\value{equation}}}, after={\setcounter{equation}{\value{enumi}}}}

\newcommand{\into}{\hookrightarrow}

\newcommand{\wtilde}{\widetilde}
\newcommand{\what}{\widehat}

\newcommand\CounterStep{\addtocounter{thm}{1}\setcounter{equation}{0}}

\newcommand{\factor}[2]{\left. \raise 2pt\hbox{$#1$} \right/\hskip -2pt\raise -2pt\hbox{$#2$}}
\newcommand{\Publication}[1]{}

\newcommand{\subversionInfo}{}
\newcommand{\svnid}[1]{}
\newcommand{\approvals}[2][Approval]{}
\usepackage{tikz-cd}

\author{Daniel Greb} %
\address{Daniel Greb, Essener Seminar für Algebraische Geometrie und Arithmetik, Fakultät für Mathe\-ma\-tik, Universität Duisburg--Essen, 45117 Essen, Germany}
\email{\href{mailto:daniel.greb@uni-due.de}{daniel.greb@uni-due.de}}
\urladdr{\href{https://www.esaga.uni-due.de/daniel.greb/}{https://www.esaga.uni-due.de/daniel.greb}}

\author{Stefan Kebekus} %
\address{Stefan Kebekus, Mathematisches Institut, Albert-Ludwigs-Universität Freiburg, Ernst-Zermelo-Straße 1, 79104 Freiburg im Breisgau, Germany \&
  Freiburg Institute for Advanced Studies (FRIAS), Freiburg im Breisgau, Germany}
\email{\href{mailto:stefan.kebekus@math.uni-freiburg.de}{stefan.kebekus@math.uni-freiburg.de}}
\urladdr{\href{https://cplx.vm.uni-freiburg.de}{https://cplx.vm.uni-freiburg.de}}

\author{Thomas Peternell} %
\address{Thomas Peternell, Mathematisches Institut, Universität
  Bayreuth, 95440~Bayreuth, Germany}
\email{\href{mailto:thomas.peternell@uni-bayreuth.de}{thomas.peternell@uni-bayreuth.de}}
\urladdr{\href{http://www.komplexe-analysis.uni-bayreuth.de}{http://www.komplexe-analysis.uni-bayreuth.de}}


\makeatletter
\@namedef{subjclassname@2020}{2020 Mathematics Subject Classification}
\makeatother

\thanks{Stefan Kebekus gratefully acknowledges partial support through a
  fellowship of the Freiburg Institute of Advanced Studies (FRIAS)}

\title[Projective flatness over klt spaces and uniformisation]{Projective flatness over klt spaces and uniformisation of varieties with nef anti-canonical divisor}
\date{\today}

\makeatletter
\hypersetup{
  pdfauthor={\authors},
  pdftitle={\@title},
  pdfsubject={\@subjclass},
  pdfkeywords={\@keywords},
  pdfstartview={Fit},
  pdfpagelayout={TwoColumnRight},
  pdfpagemode={UseOutlines},
  bookmarks,
  colorlinks,
  linkcolor=linkblue,
  citecolor=linkred,
  urlcolor=linkred
}
\makeatother

\DeclareMathOperator{\alb}{alb}
\DeclareMathOperator{\Alb}{Alb}

\DeclareMathOperator{\End}{End}
\DeclareMathOperator{\diff}{d}
\DeclareMathOperator{\Div}{Div}
\DeclareMathOperator{\GL}{GL}

\DeclareMathOperator{\PGL}{ℙ\negthinspace\GL}
\DeclareMathOperator{\Proj}{Proj}

\theoremstyle{remark}

\begin{document}

\begin{abstract}
We give a criterion for the projectivisation of a reflexive sheaf on a klt space
to be induced by a projective representation of the fundamental group of the
smooth locus.  This criterion is then applied to give a characterisation of
finite quotients of projective spaces and Abelian varieties by $ℚ$-Chern class
(in)equalities and a suitable stability condition.  This stability condition is
formulated in terms of a naturally defined extension of the tangent sheaf by the
structure sheaf.  We further examine cases in which this stability condition is
satisfied, comparing it to K-semistability and related notions.
\end{abstract}
\approvals[Approval for abstract]{Daniel & yes \\Stefan & yes\\ Thomas & yes}

\maketitle
\tableofcontents

%
% Do not edit the following line.  The text is automatically updated by
% subversion.
%
\svnid{$Id: 01-intro.tex 747 2021-04-01 07:12:38Z kebekus $}

\section{Introduction}
\subversionInfo

\subsection{Stability, the Miyaoka-Yau Inequality and quasi-étale uniformisation}
\approvals{Daniel & yes \\Stefan & yes\\ Thomas & yes}

Let $X$ be a \emph{$ℚ$-Fano $n$-fold}; that is, let $X$ be a normal, projective,
$n$-dimensional variety with at worst klt singularities such that $-K_X$ is
$ℚ$-ample.  Generalising a classical result, it has been shown in
\cite[Thm.~6.1]{GKPT19b} that if the tangent sheaf $𝒯_X$ is stable with respect
to the anticanonical polarisation $-K_X$, then its first two $ℚ$-Chern classes,
which are well-defined for all spaces with klt singularities, satisfy the
\emph{$ℚ$-Bogomolov-Gieseker Inequality},

\begin{align}
  \label{eq:BG} \frac{n-1}{2n}·\what{c}_1\left( Ω^{[1]}_X \right)² · [-K_X]^{n-2} & ≤ \what{c}_2\left( Ω^{[1]}_X \right) · [-K_X]^{n-2}.  \\
  \intertext{As part of the present investigation, we will generalise
  \eqref{eq:BG} to the case where $𝒯_X$ is semistable with respect to the
  anticanonical polarisation $-K_X$, see Section~\ref{sec:potf1} below.  In
  analogy to the case of manifolds with ample canonical bundle and as a
  generalisation of a classical result on Kähler-Einstein Fano manifolds,
  \cite{CH75}, one expects more, namely a \emph{$ℚ$-Miyaoka-Yau Inequality} of
  the form}
  \label{eq:MY} \frac{n}{2(n+1)}·\what{c}_1\left( Ω^{[1]}_X \right)² · [-K_X]^{n-2} & ≤ \what{c}_2\left( Ω^{[1]}_X \right) · [-K_X]^{n-2}.
\end{align}

\subsubsection*{The canonical extension}
\approvals{Daniel & yes \\Stefan & yes\\ Thomas & yes}

Section~\ref{ssec:nosta} shows by way of classical examples that without an
additional stability assumption \eqref{eq:MY} will not hold.  This paper
discusses an algebro-geometric (semi)stability notion that is stronger than
(semi)stability of $𝒯_X$ and guarantees that the $ℚ$-Miyaoka-Yau
Inequality~\eqref{eq:MY} holds: (semi)stability of the \emph{canonical
  extension}.  The canonical extension is a reflexive sheaf $ℰ_X$ on $X$ that
appears in the middle of the short exact sequence
\[
  0 → 𝒪_X → ℰ_X → 𝒯_X → 0
\]
whose extension class is given by the first Chern class of the $ℚ$-Cartier
divisor $-K_X$.  For the first Chern classes of line bundles over manifolds the
construction is classical, see \cite{Ati57}, and generalisations of it have
appeared in many other problems of Kähler geometry, see for instance
\cite{MR1163733, MR1165352, MR1959581, GrebWong}.  Section~\ref{ssec:canExt}
discusses the construction of the canonical extension in the singular case.
Note that the $ℚ$-Miyaoka-Yau Inequality~\eqref{eq:MY} is in fact nothing but
the $ℚ$-Bogomolov-Gieseker inequality for the canonical extension sheaf $ℰ_X$.

\subsection{Main results}
\approvals{Daniel & yes \\Stefan & yes\\ Thomas & yes}

Using the canonical extension, we may now formulate the main results.  Our point
of departure is the following.

\begin{prop}[Miyaoka-Yau Inequality and Semistable Canonical Extension]\label{prop:Fano1}
  Let $X$ be an $n$-dimensional, projective klt space.  If there exists an ample
  Cartier divisor $H$ on $X$ such that the canonical extension $ℰ_X$ is
  semistable with respect to $H$, then
  \begin{equation}\label{eq:MY_arbitrary_H}
    \frac{n}{2(n+1)}·\what{c}_1\left( Ω^{[1]}_X \right)²·[H]^{n-2} ≤ \what{c}_2\left( Ω^{[1]}_X \right)·[H]^{n-2}.
  \end{equation}
\end{prop}

Definition~\vref{defn:kltspace} recalls the notion of a klt space.  We prove
Proposition~\ref{prop:Fano1} in Section~\ref{sec:potf1}.
Section~\ref{ssec:soce} discusses criteria to guarantee stability of the
canonical extension $ℰ_X$.

\begin{rem}[Kähler-Einstein manifolds]\label{rem:smoothKE}
  If $(X, ω)$ is a Kähler-Einstein Fano manifold, Tian has shown in
  \cite[Thm.~2.1]{MR1163733} that the bundle $ℰ_X$ admits a Hermitian-Yang-Mills
  metric with respect to $ω$, and is therefore $ω$-semistable.  The Miyaoka-Yau
  Inequality~\eqref{eq:MY} then holds by \cite[Thm.~IV.4.7]{Kob87}, with
  equality if and only if the metric has constant holomorphic sectional
  curvature and $X ≅ ℙ^n$, see \cite[Thm.~IV.4.16]{Kob87} or
  \cite[Thm.~2.3]{MR1603624}.  The next result generalises this to klt varieties
  with nef anticanonical classes.
\end{rem}

\begin{thm}[Quasi-étale uniformisation if $-K_X$ is nef]\label{thm:Fano2}
  Let $X$ be a projective variety.  Assume that $X$ has at most klt
  singularities\footnote{equivalently: Assume that the pair $(X,0)$ is klt.} and
  that its anti-canonical class $-K_X$ is nef.  Then, the following statements
  are equivalent.
  \begin{enumerate}
  \item\label{il:s1} There exists an ample Cartier divisor $H$ on $X$ such that
    the canonical extension $ℰ_X$ is semistable with respect to $H$ and such
    that equality holds in \eqref{eq:MY_arbitrary_H}:
    \[
      \frac{n}{2(n+1)}·\what{c}_1\left( Ω^{[1]}_X \right)²·[H]^{n-2} =
      \what{c}_2\left( Ω^{[1]}_X \right)·[H]^{n-2}.
    \]

  \item\label{il:s2} The variety $X$ is a quotient of the projective space or of
    an Abelian variety by the action of a finite group of automorphisms that
    acts without fixed points in codimension one.
  \end{enumerate}
\end{thm}

A proof of Theorem~\ref{thm:Fano2} is given in Section~\ref{sec:potf2} below.
Section~\ref{sec:7-6} discusses many examples where equality in \eqref{eq:MY}
holds.

\begin{rem}[Torus quotients]
  If a variety $X$ satisfies Assumption~\ref{il:s1} and if additionally
  $K_X \equiv 0$, then $\what{c}_2\bigl( Ω^{[1]}_X \bigr) · [H]^{n-2} = 0$, and
  we recover the main results of \cite{GKP13, LT18}.  More intricate
  characterisations of torus quotients can be found in \cite{GKP20b}.
\end{rem}

\subsection{Application to Fano varieties}
\approvals{Daniel & yes \\Stefan & yes\\ Thomas & yes}

Note that Proposition~\ref{prop:Fano1} and Theorem~\ref{thm:Fano2} apply in
particular to $ℚ$-Fano varieties, and give a characterisation of finite
quotients of projective spaces among $ℚ$-Fano varieties $X$ with semistable
canonical extension $ℰ_X$.

In fact, in the last few decades multiple stability notions for Fano manifolds
have been introduced, with a view both towards the construction of moduli spaces
and the existence question for Kähler-Einstein metrics, see for instance the
survey \cite{Xu-K-stabilityNotes}.  The klt condition on the singularities
appears naturally in this context, see \cite{MR3010808}.  In
Section~\ref{ssec:soce}, we prove that stability of the canonical extension can
be guaranteed in natural classes of examples, which include all smooth Fano
threefolds of Picard number one.  There, we also discuss the relation to
K-(semi)stability.

Finally, together with a result of Druel-Guenancia-Păun \cite[Thm.~B]{DGP20} on
the semistability of the canonical extension on $ℚ$-Fano varieties,
Proposition~\ref{prop:Fano1} and Theorem~\ref{thm:Fano2} immediately imply the
following statement.  This generalises the classical result discussed in
Remark~\ref{rem:smoothKE} above.

\begin{thm}[Characterisation of finite quotients of projective spaces]
  Let $X$ be a $ℚ$-Fano variety admitting a singular Kähler-Einstein metric.
  Then, the $ℚ$-Miyaoka-Yau inequality \eqref{eq:MY} holds, with equality if and
  only if $X ≅ ℙ^n/G$, where $G$ is a finite group of automorphisms acting
  without fixed points in codimension one.  \qed
\end{thm}

\subsection{Projectively flat sheaves on singular spaces}
\approvals{Daniel & yes \\Stefan & yes\\ Thomas & yes}

The key objects used in our arguments are reflexive sheaves $ℱ$ such that
$ℱ|_{X_{\reg}}$ is locally free and holomorphically projectively flat; i.e., the
projectivisation of $ℱ|_{X_{\reg}}$ is induced by a projective representation of
the fundamental group of $X_{\reg}$.  Details are discussed in
Section~\ref{ssec:pfbas}.  The proofs of our main results rely on the following
technical criterion for projective flatness, which generalises classic results
from the smooth case to the setting of klt spaces.  Its proof crucially relies
on \cite{LT18}.

\begin{prop}[Criterion for projective flatness]\label{prop:pflatnessCriterion}
  Let $X$ be an $n$-dimensional, projective klt space and let $H ∈ \Div(X)$ be
  ample.  Further, let $ℱ$ be a reflexive sheaf of rank $r$ on $X$.  Assume that
  $ℱ$ is semistable with respect to $H$ and that its $ℚ$-Chern classes satisfy
  the equation
  \begin{equation}\label{eq:xxA}
    \frac{r-1}{2r}·\what{c}_1(ℱ)²·[H]^{n-2} = \what{c}_2(ℱ)·[H]^{n-2}.
  \end{equation}
  Then, $ℱ|_{X_{\reg}}$ is locally free and projectively flat.
\end{prop}

A proof of Proposition~\ref{prop:pflatnessCriterion} is given in
Section~\ref{sec:potf0} below.  Section~\ref{sec:3-3} describes the structure of
the sheaf $ℱ$ in more detail.  Further applications of
Proposition~\ref{prop:pflatnessCriterion} are discussed in \cite{GKP20b}.

\subsection{Thanks}
\approvals{Daniel & yes \\Stefan & yes\\ Thomas & yes}

We would like to thank Indranil Biswas, Ruadhaí Dervan, and Stefan Schröer for
patiently answering our questions.  Moreover, we thank Henri Guenancia and Mihai
Păun for discussions and the anonymous referee for helpful comments.

%
% Do not edit the following line.  The text is automatically updated by
% subversion.
%
\svnid{$Id: 02-notation.tex 746 2021-04-01 07:06:38Z kebekus $}

\section{Conventions and notation}
\subversionInfo

\subsection{Global conventions}
\approvals{Daniel & yes \\Stefan & yes\\ Thomas & yes}

Throughout this paper, all schemes, varieties and morphisms will be defined over
the complex number field.  We follow the notation and conventions of
Hartshorne's book \cite{Ha77}.  In particular, varieties are always assumed to
be irreducible.  We refer the reader to \cite{KM98} for notation around
higher-dimensional birational geometry.

\begin{defn}[Klt spaces]\label{defn:kltspace}
  A normal, quasi-projective variety $X$ is called a \emph{klt space} if there
  exists an effective Weil $ℚ$-divisor $Δ$ such that the pair $(X, Δ)$ is klt.
\end{defn}

\subsection{Reflexive sheaves}
\approvals{Daniel & yes \\Stefan & yes\\ Thomas & yes}

Given a normal, quasi-projective variety (or normal, irreducible complex space)
$X$, we write $Ω^{[p]}_X := \bigl(Ω^p_X \bigr)^{**}$ and refer to this sheaf as
\emph{the sheaf of reflexive differentials}.  More generally, given any coherent
sheaf $ℰ$ on $X$, write $ℰ^{[⊗ m]} := \bigl(ℰ^{⊗ m} \bigr)^{**}$ and
$\det ℰ := \bigl( Λ^{\rank ℰ} ℰ \bigr)^{**}$.  Given any morphism $f : Y → X$ of
normal, quasi-projective varieties (or normal, irreducible, complex spaces), we
write $f^{[*]} ℰ := (f^* ℰ)^{**}$.

\subsection{Varieties and complex spaces}
\approvals{Daniel & yes \\Stefan & yes\\ Thomas & yes}

In order to keep notation simple, do not distinguish between algebraic varieties
and their underlying complex spaces, unless there is specific danger of
confusion.  Along these lines, if $X$ is a quasi-projective complex variety, we
write $π_1(X)$ for the fundamental group of the associated complex space.

\subsection{Covering maps and quasi-étale morphisms}
\approvals{Daniel & yes \\Stefan & yes\\ Thomas & yes}

A \emph{cover} or \emph{covering map} is a finite, surjective morphism
$γ : X → Y$ of normal, quasi-projective varieties (or normal, irreducible
complex spaces).  The covering map $γ$ is called \emph{Galois} if there exists a
finite group $G ⊂ \Aut(X)$ such that $γ$ is isomorphic to the quotient map.

A morphism $f : X → Y$ between normal varieties (or normal, irreducible complex
spaces) is called \emph{quasi-étale} if $f$ is of relative dimension zero and
étale in codimension one.  In other words, $f$ is quasi-étale if
$\dim X = \dim Y$ and if there exists a closed, subset $Z ⊆ X$ of codimension
$\codim_X Z ≥ 2$ such that $f|_{X ∖ Z} : X ∖ Z → Y$ is étale.

\subsection{Maximally quasi-étale spaces}
\approvals{Daniel & yes \\Stefan & yes\\ Thomas & yes}
\label{ssec:mqes}

Let $X$ be a normal, quasi-projective variety (or a normal, irreducible complex
space).  We say that $X$ is \emph{maximally quasi-étale} if the natural
push-forward map of fundamental groups,
\[
  π_1(X_{\reg}) \xrightarrow{(\text{incl})_*} π_1(X)
\]
induces an isomorphism between the profinite completions,
$\what{π}_1(X_{\reg}) ≅ \what{π}_1(X)$.

\begin{rem}
  Recall from \cite[0.7.B on p.~33]{FL81} that the natural push-forward map
  $(\text{incl})_*$ is always surjective.  If $X$ is any klt space, then $X$
  admits a quasi-étale cover that is maximally quasi-étale and again a klt
  space, \cite[Thm.~1.14]{GKP13}.
\end{rem}

%
% Do not edit the following line.  The text is automatically updated by
% subversion.
%
\svnid{$Id: 03-projectiveFlatness.tex 747 2021-04-01 07:12:38Z kebekus $}

\section{Projective flatness}
\subversionInfo

\subsection{Projectively flat bundles and sheaves}
\approvals{Daniel & yes \\Stefan & yes\\ Thomas & yes}
\label{ssec:pfbas}

As pointed out in the introduction, \emph{projective flatness} is the core
concept of this paper.  Projectively flat bundles over differentiable manifolds
are thoroughly discussed in the literature, for instance in the classic textbook
\cite{Kob87}.  We are, however, not aware of references that cover the singular
case.  We have therefore chosen to introduce the relevant notions in some detail
here.  In a nutshell, we call a projective space bundle \emph{projectively flat}
if it comes from a $\PGL$-representation of the fundamental group.  The
following construction and the subsequent definitions make this precise.

\begin{construction}\label{cons:2-1}
  Let $X$ be a normal and irreducible complex space.  Given a number $r ∈ ℕ$ and
  a representation of the fundamental group, $ρ : π_1(X) → \PGL(r+1, ℂ)$,
  consider the universal cover $\wtilde{X} → X$ and the diagonal action of
  $π_1(X)$ on $\wtilde{X} ⨯ ℙ^r$.  The quotient
  \[
    ℙ_ρ := \factor{\wtilde{X} ⨯ ℙ^r}{π_1(X)}
  \]
  is a complex space that carries the natural structure of a locally trivial
  $ℙ^r$-bundle over the original space $X$.
\end{construction}

\begin{defn}[Projectively flat bundles and sheaves on complex spaces]\label{def:3-2}
  Let $X$ be a normal and irreducible complex space, let $r ∈ ℕ$ be any number
  and let $ℙ → X$ be a locally trivial $ℙ^r$-bundle.  We call the bundle $ℙ → X$
  \emph{(holomorphically) projectively flat} if there exists a representation of
  the fundamental group, $ρ : π_1(X) → ℙ\negthinspace\GL(r+1,ℂ)$, and an
  isomorphism $ℙ ≅_X ℙ_ρ$ of complex spaces over $X$, where $ℙ_ρ$ is the bundle
  constructed in \ref{cons:2-1} above.  A locally free coherent sheaf $ℱ$ of
  $𝒪_X$-modules is called \emph{(holomorphically) projectively flat} if the
  associated bundle $ℙ(ℱ)$ is projectively flat.
\end{defn}

\begin{defn}[Projectively flat bundles and sheaves on complex varieties]
  Let $X$ be a connected, complex, quasi-projective variety and let $r ∈ ℕ$ be
  any number.  An étale locally trivial $ℙ^r$-bundle $ℙ → X$ is called
  \emph{projectively flat} if the associated analytic bundle
  $ℙ^{(an)} → X^{(an)}$ is projectively flat.  Ditto for coherent sheaves.
\end{defn}

On complex manifolds, projective flatness is of course equivalent to the
existence of certain connections.  We briefly recall the following standard
fact.

\begin{fact}[Projective flatness and connections]\label{fact:projFlat}
  Let $X$ be a connected complex manifold and let $ℱ$ be a locally free coherent
  sheaf on $X$.  Then, the following are equivalent.
  \begin{enumerate}
  \item The locally free sheaf $ℱ$ is projectively flat in the sense of
    Definition~\ref{def:3-2}.

  \item The locally free sheaf $ℱ$ admits a holomorphic connection whose
    curvature tensor is of the form
    \[
      R = α · \Id_ℱ ∈ H⁰\bigl(X, \,Ω²_X ⊗ \sEnd(ℱ) \bigr)
    \]
    for some holomorphic $2$-form $α$ on $X$.
  \end{enumerate}
\end{fact}
\begin{proof}
  In the $C^∞$-setting, this is \cite[I.Cor.~2.7 and I.Prop.~2.8]{Kob87}.  The
  proofs carry over to the holomorphic setting \emph{mutatis mutandis}.
\end{proof}

\subsection{Projective flatness and flatness}
\approvals{Daniel & yes \\Stefan & yes\\ Thomas & yes}

In our earlier paper \cite{GKP13} we discussed locally free sheaves $ℰ$ on
singular spaces $X$ that were flat in the sense that $ℰ$ was defined by a
representation $π_1(X) → \GL(r,ℂ)$, in a manner analogous to
Construction~\ref{cons:2-1}.  The two notions are of course related.

\begin{prop}[Projective flatness and flatness of derived sheaves]\label{prop:4a}
  Let $X$ be a normal and irreducible complex space and $ℰ$ a rank-{}$r$,
  locally free sheaf on $X$.  If $ℰ$ is projectively flat, then the locally free
  sheaves $\sEnd(ℰ)$ and $\Sym^r ℰ ⊗ \det ℰ^*$ are flat in the sense of
  \cite[Def.~2.13]{GKPT19b}.
\end{prop}
\begin{proof}
  The group morphisms
  \[
    \begin{matrix}
      \GL(r, ℂ) & → & \GL\bigl( r², ℂ \bigr) \\
      A & ↦ & \End A
    \end{matrix}
    \qquad\text{and}\qquad
    \begin{matrix}
      \GL(r, ℂ) & → & \GL\bigl( \left( \begin{smallmatrix} 2r-1 \\ r
        \end{smallmatrix}\right), ℂ \bigr) \\
      A & ↦ & \frac{1}{\det A}·\Sym^r A
    \end{matrix}
  \]
  factor via $\PGL(r, ℂ)$.
\end{proof}

In case where $X$ is maximally quasi-étale, Proposition~\ref{prop:4a} has the
following consequence, which we find remarkable because it can be used to
guarantee local freeness of certain sheaves at the singular points of $X$.  We
will later use it in a setting where $ℰ$ is constructed so that $\sEnd(ℰ)$
contains the tangent sheaf $𝒯_X$ as a direct summand and where Chern class
equalities guarantee projective flatness of $ℰ|_{X_{\reg}}$.

\begin{cor}[Projective flatness and local freeness of derived sheaves I]\label{cor:3-9}
  Let $X$ be a normal, irreducible complex space that is maximally quasi-étale.
  Let $ℰ$ be a reflexive sheaf on $X$ such that $ℰ|_{X_{\reg}}$ is locally free
  and projectively flat.  Then, the sheaves
  $\bigl(\Sym^r ℰ ⊗ \det ℰ^*\bigr)^{**}$ and $\sEnd(ℰ)$ are both locally free
  and flat in the sense of \cite[Def.~2.13]{GKPT19b}.
\end{cor}
\begin{proof}
  The assumption that $X$ is maximally quasi-étale says that the natural
  morphism of étale fundamental groups,
  \[
    \what{π}_1 \bigl( X_{\reg} \bigr) → \what{π}_1 \bigl( X \bigr),
  \]
  is isomorphic.  The arguments from \cite[proof of Thm.~1.14 on p.~26]{GKP13}
  now apply verbatim to yield the desired extension.  Alternatively, use the
  fact that any finite-dimensional complex representation of
  $π_1 \bigl( X_{\reg} \bigr)$ extends to a representation of
  $π_1 \bigl( X \bigr)$ by \cite[Thm.~1.2b]{MR0262386}.
\end{proof}

Once we are working over a complex manifold rather than over an arbitrary
complex space, the converse statements are also true.

\begin{prop}[Projective flatness and local freeness of derived sheaves II]\label{prop:3-5}
  Let $X$ be a connected complex manifold and $ℱ$ be a locally free coherent
  sheaf on $X$ of rank $r$.  Then, the following are equivalent.
  \begin{enumerate}
  \item\label{il:proj_flat} The locally free sheaf $ℱ$ is projectively flat in
    the sense of Definition~\ref{def:3-2}.
    
  \item\label{il:End_flat} The locally free sheaf $\sEnd(ℱ)$ is flat in the
    sense that it admits a flat holomorphic connection.
    
  \item\label{il:sym} The locally free sheaf $(\Sym^r ℱ) ⊗ (\det ℱ^*)$ is flat
    in the sense that it admits a flat holomorphic connection.
  \end{enumerate}
\end{prop}
\begin{proof}
  The implication ``\ref{il:proj_flat}$\, ⇒\, $\ref{il:End_flat}'' follows from
  Fact~\ref{fact:projFlat}.  We refer the reader to \cite[I.Cor.~2.7 and
  I.Prop.~2.8]{Kob87} for a proof of the analogous statement in the
  $C^∞$-setting; again, the arguments carry over to the holomorphic setting
  \emph{mutatis mutandis}.  The implication
  ``\ref{il:End_flat}$\, ⇒\, $\ref{il:proj_flat}'' is proven in \cite[proof of
  Prop.~2.1]{MR2545454}, crucially using the fact that $\End(V)$ is a faithful
  representation of the complex reductive group $\PGL(V)$, where $V$ is a
  complex vector space of dimension $r$.  The equivalence
  ``\ref{il:proj_flat}$\, ⇔\, $\ref{il:sym}'' is proven in an entirely similar
  fashion, by looking at the representation $\Sym^r(V) ⊗ \det(V)$ instead of
  $\End(V)$.  We leave it to the reader to spell out the details.
\end{proof}

\begin{rem}
  For more information on locally free sheaves with a flat holomorphic
  connection, the reader is referred to \cite[Sect.~7]{Ati57}.
\end{rem}

The following example shows that under the assumptions of
Corollary~\ref{cor:3-9} the sheaf $ℰ$ itself is not necessarily locally free.

\begin{example}
  Let $q: \wtilde{X} → ℙ¹$ be the second Hirzebruch surface with its natural
  projection to $ℙ¹$, that is,
  $\wtilde{X} := ℙ_{ℙ¹}\bigl(𝒪_{ℙ¹}(2) ⊕ 𝒪_{ℙ¹}\bigr)$.  Let $E ⊂ \wtilde{X}$ be
  the section at infinity, and note that the complement $\wtilde{X}∖E$ is simply
  connected, since it is the total space of a line bundle over a simply
  connected space.  Next, consider the blow-down $π: \wtilde{X} → X$ of $E$.
  The variety $X$ is then isomorphic to the subvariety of $ℙ³$ given as the cone
  over a quadric normal curve in $ℙ²$.  Finally, let $D := π(F)$ be the image of
  a $q$-fibre $F$.  The variety $D$ is then a line through the unique singular
  point $P ∈ X$.  Since $π$ is an isomorphism away from $E$, the smooth locus
  $X_{\reg}$ of $X$ is simply connected; in particular, $X$ is maximally
  quasi-étale.  The Weil divisorial sheaf $ℰ := 𝒪_X(D)$ is not locally free at
  $P$; however, $\sEnd (ℰ) ≅ \bigl(𝒪_X(D) ⊗ 𝒪_X(-D)\bigr)^{**} ≅ 𝒪_X$ is locally
  free and flat on $X$.
\end{example}

\subsection{Projective flatness on maximally quasi-étale spaces}
\approvals{Daniel & yes\\Stefan & yes\\ Thomas & yes}
\label{sec:3-3}

The proofs of our main results require that we consider a setting where $X$ is
normal and maximally quasi-étale, and $ℱ$ is a coherent, reflexive sheaf on $X$
whose restriction $ℱ|_{X_{\reg}}$ is locally free and projectively flat in the
sense of Definition~\ref{def:3-2}.

\subsubsection{Extension and local structure}
\approvals{Daniel & yes \\Stefan & yes\\ Thomas & yes}

In this section we show that the structure of $ℱ$ at the singular points of $X$
is eventually rather simple.

\begin{prop}[Extension of projectively flat $ℙ^n$-bundles]\label{prop:projflat}
  Let $X$ be a normal, irreducible complex space that is maximally quasi-étale.
  Then, any representation $ρ°: π_1(X_{\reg}) → \PGL(r,ℂ)$ factors via a
  representation of $π_1(X)$ as follows,
  \[
    \begin{tikzcd}[column sep=large]
      π_1(X_{\reg}) \arrow[rr, bend left=15, "ρ°"] \arrow[r, two heads, "(\text{incl})_*"'] & π_1(X) \arrow[r, "ρ"'] & ℙ\negthinspace\GL(r,ℂ).
    \end{tikzcd}
  \]
  In particular, any projectively flat $ℙ^{r-1}$-bundle $ℙ_{\reg} → X_{\reg}$
  extends to a projectively flat projective bundle $ℙ_X → X$.
\end{prop}
\begin{proof}
  Write $G := \Image(ρ°)$.  Since $\PGL(r,ℂ)$ is a linear group, so is $G$.  In
  particular, it follows from Malcev's theorem, \cite[Thm.~4.2]{MR0335656}, that
  $G$ is residually finite.  As in the proof of Corollary~\ref{cor:3-9}, using
  the arguments from \cite[proof of Thm.~1.14 on p.~26]{GKP13} we find that the
  representation $ρ°$ is induced by a representation of $π_1 \bigl( X \bigr)$,
  which yields the desired extension.
\end{proof}

Proposition~\ref{prop:projflat} has further consequences.  While we cannot
expect that the sheaf $ℱ$ is locally free near the singular points of $X$, it
turns out that $ℱ$ locally (in the analytic topology) always looks like a direct
sum of $\rank ℱ$ copies of the same Weil divisorial sheaf.

\begin{prop}[Local description of projectively flat sheaves]\label{prop:3-3x}
  Let $X$ be a normal, irreducible complex space that is maximally quasi-étale.
  Let $ℱ_{X_{\reg}}$ be a locally free and projectively flat coherent sheaf on
  $X_{\reg}$.  Then, the following holds.
  \begin{enumerate}
  \item\label{il:33x-1} If $U ⊆ X$ is any simply connected open subset, then
    there exists an invertible sheaf $ℒ_{U_{\reg}} ∈ \Pic(U_{\reg})$, unique up
    to isomorphism, such that
    \[
      ℱ_{X_{\reg}}|_{U_{\reg}} ≅ (ℒ_{U_{\reg}})^{⊕ \rank ℱ_{X_{\reg}}}.
    \]
    
  \item\label{il:33x-2} The sheaf $ℱ_{X_{\reg}}$ extends from $X_{\reg}$ to a
    reflexive coherent sheaf $ℱ_X$ on $X$ if and only if for every simply
    connected open $U ⊆ X$, the sheaf $ℒ_{U_{\reg}}$ extends from $U_{\reg}$ to
    a reflexive coherent sheaf $ℒ_U$ of rank one on $U$.  In this case,
    \[
      ℱ_X|_U ≅ (ℒ_U)^{⊕ \rank ℱ_{X_{\reg}}}.
    \]
  \end{enumerate}
\end{prop}
\begin{proof}
  Write $r := \rank ℱ$.  We have seen in Proposition~\ref{prop:projflat} that
  the projectively flat $ℙ^{r-1}$-bundle $ℙ_{X_{\reg}} := ℙ(ℱ_{X_{\reg}})$
  extends to a projectively flat $ℙ^{r-1}$-bundle $ℙ_X$ on $X$.  By
  construction, the restriction of $ℙ_X$ to any simply connected open subset
  $U ⊆ X$ is trivial and then so is its restriction to $U_{\reg}$.  In
  particular, we have an isomorphism
  \[
    ℙ(ℱ_{X_{\reg}})|_{U_{\reg}} ≅_{U_{\reg}} ℙ_X|_{U_{\reg}} %
    ≅_{U_{\reg}} ℙ_U( 𝒪_{U_{\reg}}^{⊕ r+1}).
  \]
  It follows that $ℱ_{X_{\reg}}|_{U_{\reg}}$ and $𝒪_{U_{\reg}}^{⊕ r+1}$ differ
  only by a twist with an invertible $ℒ_{U_{\reg}} ∈ \Pic(U_{\reg})$.  This
  shows \ref{il:33x-1}.
  
  To prove \ref{il:33x-2}, consider one of the open sets $U$ and write
  $ι : U_{\reg} → U$ for the obvious inclusion.  Since push-forward respects
  direct summands, we find that
  \[
    ι_* (ℱ_{X_{\reg}}|_{U_{\reg}}) ≅ (ι_* ℒ_{U_{\reg}})^{⊕ r},
  \]
  and the left side is coherent if and only if every summand of the right side
  is.
\end{proof}

\subsubsection{Pull-back}
\approvals{Daniel & yes \\Stefan & yes\\ Thomas & yes}

As one consequence of the local description of projectively flat sheaves, we
find that they satisfy a surprising pull-back property.  We believe that the
following corollary is of independent interest and include it here for later
reference, even though to do not use it in the sequel.

\begin{cor}[Pull-back of projectively flat sheaves]\label{cor:3-10}
  Let $X$ be a normal, irreducible complex space that is maximally quasi-étale.
  Let $ℱ$ be a reflexive sheaf on $X$ such that $ℱ|_{X_{\reg}}$ is locally free
  and projectively flat.  If $\varphi : Y → X$ is any morphism where $Y$ is
  irreducible, smooth and where $\img(\varphi) ⊄ X_{\sing}$, then the reflexive
  pull-back $\varphi^{[*]} ℱ$ is locally free and projectively flat.
\end{cor}

\begin{rem}
  Note that Corollary~\ref{cor:3-10} applies to settings where the preimage
  $\varphi^{-1}(X_{\sing})$ is a divisor in $Y$.  One relevant case is where
  $\varphi$ is the inclusion map of a smooth curve $Y ⊂ X$ that passes through
  the singular locus.
\end{rem}

\begin{proof}[Proof of Corollary~\ref{cor:3-10}]
  For convenience of notation, write $ℱ_Y := \varphi^{[*]} ℱ$.  We begin by
  showing that $ℱ_Y$ is locally free.  As this question is local over $X$, we
  may assume without loss of generality that $X$ is simply connected.
  Proposition~\ref{prop:3-3x} will then allow to find a Weil divisorial sheaf
  $ℒ$ on $X$ such that $ℱ = ℒ^{⊕ \rank ℱ}$, which implies
  $ℱ_Y = \bigl( \varphi^{[*]} ℒ \bigr)^{⊕ \rank ℱ}$.  Local freeness of $ℱ_Y$
  follows since $\varphi^{[*]} ℒ$ is Weil divisorial on the smooth space $X$ and
  hence invertible.

  It remains to show that $ℱ_Y$ is projectively flat.  By
  Proposition~\ref{prop:3-5} it suffices to show that $\sEnd(ℱ_Y)$ is flat.
  This will be established by showing that the natural morphism
  \begin{equation}\label{eq:6456}
    \varphi^* \sEnd(ℱ) → \sEnd(ℱ_Y)
  \end{equation}
  is isomorphic.  But we have already seen in Corollary~\ref{cor:3-9} that
  $\sEnd(ℱ)$ is locally free and flat, and then so is its pull-back
  $\varphi^* \sEnd(ℱ)$.  To be more precise, recall that the
  Morphism~\eqref{eq:6456} comes to be as a composition
  \begin{equation}\label{eq:6457}
    \begin{tikzcd}
      \varphi^* \sEnd(ℱ) \ar[r] & \sEnd(\varphi^* ℱ) \ar[r] & \sEnd(\varphi^{[*]} ℱ).
    \end{tikzcd}
  \end{equation}
  The existence of the first morphism in \eqref{eq:6457} follows from the
  observation that any endomorphism of $ℱ$ induces an endomorphism of the
  pull-back $\varphi^* ℱ$.  The existence of the second morphism in
  \eqref{eq:6457} follows from the observation that any endomorphism of any
  sheaf induces an endomorphism morphism of its reflexive hull.

  The question ``Is \eqref{eq:6456} isomorphic?'' is again local over $X$, so
  that we may again assume without loss of generality that $ℱ = ℒ^{⊕ \rank ℱ}$.
  As before, we consider the invertible sheaf
  $ℒ_Y := \varphi^{[*]} ℒ ∈ \Pic(Y)$.  As both arrows in \eqref{eq:6457} respect
  the direct sum decomposition, it suffices to show that the natural morphism
  \begin{equation}\label{eq:6458}
    𝒪_Y ≅ \varphi^* 𝒪_X ≅ \varphi^* \sEnd(ℒ) → \sEnd(ℒ_Y) ≅ 𝒪_Y
  \end{equation}
  is isomorphic.  The assumption $\img(\varphi) ⊄ X_{\sing}$ implies that
  \eqref{eq:6458} is generically injective, hence injective.  But
  \eqref{eq:6458} is also surjective, as the pull-back of the identity section
  $\Id_{ℒ} ∈ \sEnd(ℒ)(X)$, which generates $\sEnd(ℒ)$, maps to the identity
  $\Id_{ℒ_Y} ∈ \sEnd(ℒ_Y)(Y)$, which generates $\sEnd(ℒ_Y)$.
\end{proof}

%
% Do not edit the following line.  The text is automatically updated by
% subversion.
%
\svnid{$Id: 04-canonicalExtension.tex 747 2021-04-01 07:12:38Z kebekus $}

\section{The canonical extension}
\label{ssec:canExt}
\subversionInfo
\approvals{Daniel & yes \\Stefan & yes\\ Thomas & yes}

The following construction is classical and well-known in the smooth setup, see
\cite{Ati57}, and has been studied for various questions in Kähler geometry, for
example in \cite{MR1163733, MR1959581, GrebWong}.

\begin{construction}[Extensions induced by $ℚ$-Cartier divisors]\label{cons:DExt}
  Let $X$ be a normal variety and $D$ be a $ℚ$-Cartier Weil divisor on $X$.
  Choose any integer $m ∈ ℕ^{>0}$ such that $m·D$ is Cartier, and let
  $c_1\bigl(𝒪_X(m·D)\bigr)$ in $H¹\bigl(X,\, Ω¹_X \bigr)$ be the first Chern
  class of the locally free sheaf $𝒪_X(m·D)$.  Spelled out: if
  $\mathscr{U} = (U_α)_{α ∈ I}$ is a trivialising covering for $𝒪_X(m·D)$ with
  transition functions $g_{αβ} ∈ 𝒪^*_X(U_{αβ})$, then the first Chern class
  $c_1 \bigl(𝒪_X(m·D)\bigr)$ is the image of the Čech cohomology class
  \[
    \Bigl[\Bigl(\frac{dg_{αβ}}{g_{αβ}}\Bigr)_{α, β ∈ I}\Bigr] %
    ∈ \check{H}¹\bigl(\mathscr{U}, \, Ω¹_X \bigr)
  \]
  in $H¹\bigl(X,\, Ω¹_X \bigr)$.  We can then assign to $D$ the cohomology class
  \[
    c_1(D) := \frac{1}{m}·c_1\bigl(𝒪_X(m·D)\bigr) ∈ H¹\bigl(X,\, Ω¹_X \bigr).
  \]
  Observe that the class $c_1(D)$ is independent of the choice of $m$.  Using
  the canonical identification
  $H¹\bigl(X,\, Ω¹_X \bigr) ≅ \Ext¹\bigl(𝒪_X,\, Ω¹_X \bigr)$ and the standard
  interpretation of first Ext-groups, \cite[III.Prop.~6.3]{Ha77} and
  \cite[Ex.~A.3.26]{E95}, in this way we obtain the isomorphism class of an
  extension
  \begin{equation}\label{eq:OmegaExtension}
    0 → Ω_X¹ → 𝒲_D → 𝒪_X → 0.
  \end{equation}
  We refer to \eqref{eq:OmegaExtension} as the \emph{extension of $Ω_X¹$ by
    $𝒪_X$ induced by $D$}.  Functoriality of Ext-groups with respect to
  restriction to open subsets implies that \eqref{eq:OmegaExtension} splits on
  any affine open subset of $X$.  In particular, we find that the extension
  defined by $D$ is locally splittable.  This allows to dualise
  \eqref{eq:OmegaExtension} in order to obtain a likewise locally splittable
  extension of $𝒯_X$ by $𝒪_X$,
  \begin{equation}\label{eq:tangExtension}
    0 → 𝒪_X → ℰ_D → 𝒯_X → 0.
  \end{equation}
\end{construction}

\begin{rem}[Reflexivity of $ℰ_D$]\label{rem:refl}
  In the setting of Construction~\ref{cons:DExt}, recall that $𝒯_X$ is
  reflexive.  This implies in particular that the middle term $ℰ_D$ of the
  locally splittable Extension~\eqref{eq:tangExtension} is likewise reflexive.
  Alternatively, use that $ℰ_D ≅ \sHom(𝒲_D, 𝒪_X)$ to reach the same conclusion.
\end{rem}

\begin{rem}[Functoriality in morphisms]\label{rem:func}
  In the setting of Construction~\ref{cons:DExt}, if $γ : Y → X$ is any morphism
  of normal varieties, then $c_1(γ^*D)$ is the image of $c_1(D)$ under the
  obvious map
  \[
    H¹\bigl( \diff γ \bigr) : H¹\bigl(X,\, Ω¹_X \bigr) → H¹\bigl(Y,\, Ω¹_Y
    \bigr).
  \]
  If $γ$ is étale, this implies that $𝒲_{γ^*D} ≅ γ^* 𝒲_D$ and
  $ℰ_{γ^*D} ≅ γ^* ℰ_D$.
\end{rem}

\begin{rem}[Geometric realisation]
  Over $X_{\reg}$, where $𝒪_X(D)$ is locally free with associated line bundle
  $L → X_{\reg}$, the sequence \eqref{eq:tangExtension} is just the Atiyah
  extension \cite[Thm.~1]{Ati57}.  It coincides with the natural exact sequence
  \[
    0 → 𝒪_{X_{\reg}} \overset{e}{→} 𝒯_L(-\log X_{\reg})|_{X_{\reg}} →
    𝒯_{X_{\reg}} → 0,
  \]
  where we identified $X_{\reg}$ with the zero section of $L$, and $e$ is given
  by the vector field coming from the natural $ℂ^*$-action; see also
  \cite[p.~407]{MR1163733}.  Similarly, if $D$ is Cartier with associated line
  bundle $L$ over $X$, then \eqref{eq:tangExtension} is
  \[
    0 → 𝒪_X \overset{e}{→} 𝒯_L(-\log X)|_X → 𝒯_X → 0.
  \]
  To see this, use the fact that $L → X$ is locally trivial to see that $L$ is
  normal, so that $𝒯_L(-\log X)$ is reflexive, and then again to obtain that
  $𝒯_{L}(-\log X)|_X$ stays reflexive.
\end{rem}

\begin{defn}[The canonical extension]\label{defn:tce}
  If $X$ is a normal projective variety such that $K_X$ is $ℚ$-Cartier, we apply
  Construction~\ref{cons:DExt} to the divisor $D = -K_X$, in order to obtain a
  locally splittable extension with reflexive middle term,
  \begin{align}\label{eq:canonicalExtension}
    0 → 𝒪_X → ℰ_X → 𝒯_X → 0.
  \end{align}
  We refer to \eqref{eq:canonicalExtension} as the \emph{canonical extension of
    $𝒯_X$ by $𝒪_X$}.  Abusing language somewhat, we will also refer to the sheaf
  $ℰ_X$ as the \emph{canonical extension}.
\end{defn}

\begin{rem}[Functoriality in morphisms]\label{rem:func2}
  In the setting of Definition~\ref{defn:tce}, given any quasi-étale morphism
  $γ : Y → X$ of normal varieties, then $ℰ_Y ≅ γ^{[*]} ℰ_X$ -- restrict to
  $X_{\reg}$, use Remarks~\ref{rem:refl} and \ref{rem:func}, as well as the fact
  that two reflexive sheaves agree if they agree on a big open set.
\end{rem}

\begin{rem}[The $ℚ$-Fano case]\label{rem:QFngs}
  If $X$ is a $ℚ$-Fano variety, then by definition $-K_X$ is $ℚ$-ample and as a
  consequence, the canonical extension is highly non-trivial.  In particular, if
  $C ⊂ X_{\reg}$ is any smooth projective curve $C$ inside the smooth locus of
  $X$, then the restricted sequence
  \[
    0 → 𝒪_C → ℰ_X|_C → 𝒯_X|_C → 0
  \]
  is exact and not globally splittable.
\end{rem}

%
% Do not edit the following line.  The text is automatically updated by
% subversion.
%
\svnid{$Id: 05-proofs.tex 747 2021-04-01 07:12:38Z kebekus $}

\section{Proofs of the main results}
\subversionInfo

\subsection{Proof of Proposition~\ref*{prop:pflatnessCriterion} (``Criterion for projective flatness'')}
\label{sec:potf0}
\approvals{Daniel & yes\\Stefan & yes\\ Thomas & yes}

We consider the endomorphism sheaf $ℰ := \sEnd(ℱ, ℱ)$.  The sheaf $ℰ$ is then
likewise reflexive, semistable with respect to $H$ by \cite[Prop.~4.4]{GKP15},
and its first $ℚ$-Chern class vanishes.  Equation~\eqref{eq:xxA} will then
guarantee that
\[
  \what{ch}_1(ℰ)·[H]^{n-1} = 0 \quad\text{and}\quad \what{ch}_2(ℰ)·[H]^{n-2} =
  0.
\]
Indeed, vanishing of the first term has already been mentioned above.  Regarding
the second term we therefore observe that by \cite[Thm.~3.13.2]{GKPT19b} it
suffices to show that $\what{c}_2(ℰ|_S) = 0$, where $S$ is a general, and hence
klt, complete intersection surface for a suitable multiple of $H$.  We notice
that for such a surface $S$ we have $ℰ|_S ≅ \sEnd(ℱ|_S, ℱ|_S)$ and that, by
\cite[Thm.~3.13.2]{GKPT19b} again, equation~\eqref{eq:xxA} implies that
$\what{Δ}(ℱ|_S) = 0$, where $\what{Δ}$ is the $ℚ$-Bogomolov discriminant, see
\cite[Def.~3.15]{GKPT19b}.  The standard calculus for Chern classes of vector
bundles together with \cite[Thm.~3.13.1]{GKPT19b}\footnote{Especially, see the
  preprint version of \emph{loc.~cit.}, where the construction of
  $\mathbb{Q}$-Chern classes is carried out in detail.} then implies that
$\what{c}_2(ℰ|_S) = \what{Δ} (ℱ|_S) = 0$, as desired.  We are hence in the
position to apply \cite[Thm.~1.4]{LT18} to see that $ℰ|_{X_{\reg}}$ is locally
free and carries a flat holomorphic connection; Proposition~\ref{prop:3-5} then
gives the claim.  \qed

\subsection{Proof of Proposition~\ref*{prop:Fano1} (``Miyaoka-Yau Inequality and Semistable Canonical Extension'')}
\label{sec:potf1}
\approvals{Daniel & yes \\Stefan & yes\\ Thomas & yes}

We have already remarked that Inequality~\eqref{eq:MY_arbitrary_H} is equivalent
to the $ℚ$-Bogomolov-Gieseker Inequality for the reflexive sheaf $ℰ_X$,
\begin{equation}\label{eq:B1}
  \what{Δ}(ℰ_X)·[H]^{n-2} ≥ 0
\end{equation}
where $\what{Δ}$ is the $ℚ$-Bogomolov discriminant, as introduced in
\cite[Def.~3.15]{GKPT19b}.  To begin the proof in earnest, let $m ∈ ℕ$ be
sufficiently large such that $m·H$ is very ample, let
\[
  (H_1, …, H_{n-2}) ∈ |m·H|^{⨯ (n-2)}
\]
be a general tuple of hypersurfaces, and let $S := H_1 ∩ ⋯ ∩ H_{n-2}$ be the
associated complete intersection surface in $X$.  Recall that…
\begin{itemize}
\item the surface $S$ is a klt space, \cite[Lem.~5.17]{KM98},
     
\item the restricted sheaf $ℰ_S := ℰ_X|_S$ is reflexive,
  \cite[Thm.~12.2.1]{EGA4-3},

\item the sheaf $ℰ_S$ is semistable with respect to $H_S := H|_S$,
  \cite[Thm.~1.2]{Flenner84}, and

\item Inequality~\eqref{eq:B1} is equivalent to $\what{Δ}(ℰ_S) ≥ 0$,
  \cite[Thm.~3.13.2]{GKPT19b}.
\end{itemize}
Next, recall from \cite[Thm.~3.13.(1)]{GKPT19b} that there exists a normal
surface $\what{S}$ and a finite cover $γ : \what{S} → S$ such that…
\begin{itemize}
\item the reflexive pull-back $ℰ_{\what{S}} := γ^{[*]} ℰ_S$ is locally free, and

\item the desired Inequality $\what{Δ}(ℰ_S) ≥ 0$ is equivalent to
  $Δ(ℰ_{\what{S}}) ≥ 0$, where $Δ$ is the standard Bogomolov discriminant.
\end{itemize}
Also, recall from \cite[Lem.~3.2.2]{MR2665168} that $ℰ_{\what{S}}$ is semistable
with respect to the ample divisor $H_{\what{S}} := γ^*H$.  Finally, let
$π : \wtilde{S} → \what{S}$ be a resolution of singularities.  Then,
$ℰ_{\wtilde{S}} := π^* ℰ_{\what{S}}$ is semistable with respect to the nef
divisor $H_{\wtilde{S}} := π^* H_{\what{S}}$, \cite[Prop.~2.7 and
Rem.~2.8]{GKP15} and $Δ(ℰ_{\wtilde{S}}) = Δ(ℰ_{\what{S}})$.  But
$Δ(ℰ_{\what{S}}) ≥ 0$ by \cite[Thm.~5.1]{GKP15}, which ends the proof of
Proposition~\ref{prop:Fano1}.  \qed

\subsection{Proof of Theorem~\ref*{thm:Fano2} (``Quasi-étale uniformisation if $-K_X$ is nef'')}
\label{sec:potf2}
\approvals{Daniel & yes\\Stefan & yes\\ Thomas & yes}

To prepare for the proof, observe that if $γ: Y → X$ is any quasi-étale cover,
then the following holds.\CounterStep
\begin{enumerate}
\item\label{il:t1} Recall from \cite[Prop.~5.20]{KM98} that $Y$ is again klt and
  notice that $-K_Y = γ^*(-K_X)$ is also $ℚ$-nef.
  
\item\label{il:t2} The calculus of $ℚ$-Chern classes, \cite[Lem.~3.16]{GKPT19b},
  shows that equality holds in the $ℚ$-Miyaoka-Yau inequality for $Y$ (with
  respect to the ample divisor $γ^*H$) if and only if it holds for $X$ (with
  respect to $H$).
  
\item\label{il:t3} We have seen in Remark~\ref{rem:func2} that
  $ℰ_Y ≅ γ^{[*]} ℰ_X$.  Recalling that the reflexive sheaf $γ^{[*]} ℰ_X$ is
  semistable with respect to $γ^* H$ if and only if the sheaf $γ^* ℰ_X/\tor$ is
  semistable with respect to $γ^* H$, it follows from
  \cite[Lem.~3.2.2]{MR2665168} that $ℰ_Y$ is semistable with respect to $γ^* H$
  if and only if $ℰ_X$ is semistable with respect to $H$.
\end{enumerate}

\subsection*{Direction \ref{il:s2} $⇒$ \ref{il:s1}}
\approvals{Daniel & yes \\Stefan & yes\\ Thomas & yes}

Assume that $X$ is of the form $X = Y/G$, where $Y$ is the projective space or
an Abelian variety, and where $G < \Aut(Y)$ is a finite group acting fixed-point
free in codimension one.  Denote the quotient map by $γ : Y → X$, observe that
$γ$ is quasi-étale and let $H ∈ \Div(X)$ be any ample Cartier divisor.  No
matter whether $Y$ is $ℙ^n$ or $Y$ is Abelian, it is well-known in either case
that the quotient space $X$ has klt singularities, that equality holds in the
Miyaoka-Yau Inequality for $γ^*H$ and that the canonical extension $ℰ_Y$ is
semistable with respect to $γ^*H$.  Since $γ$ is quasi-étale,
Items~\ref{il:t1}--\ref{il:t3} will then imply the same for $X$ and for the
divisor $H$.

\subsection*{Direction \ref{il:s1} $⇒$ \ref{il:s2}}
\approvals{Daniel & yes \\Stefan & yes\\ Thomas & yes}
  
Assume that $X$ satisfies the assumptions in \ref{il:s1} and choose a Galois,
maximally quasi-étale cover $γ : Y → X$ is, as provided by
\cite[Thm.~1.5]{GKP13}.  Items~\ref{il:t1}--\ref{il:t3} guarantee that $Y$
reproduces the Assumptions~\ref{il:s1}.  Replacing $X$ by $Y$, we may (and will)
therefore assume from now on that $X$ itself is maximally quasi-étale.  With
these additional assumptions, we aim to show that $X$ is isomorphic to $ℙ^n$ or
to an Abelian variety, from which the main claim follows.

As $ℰ_X$ is $H$-semistable and equality holds in \eqref{eq:MY_arbitrary_H},
Proposition~\ref{prop:pflatnessCriterion} (``Criterion for projective
flatness'') implies that $ℰ_X|_{X_{\reg}}$ is projectively flat.  As a
consequence, we may apply Corollary~\ref{cor:3-9} and find that $\sEnd(ℰ_X)$ is
locally free and flat.  By construction of $ℰ_X$ as a locally splittable
extension, we can then locally write
\[
  ℰ_X = 𝒪_X ⊕ 𝒯_X \quad \text{and} \quad \sEnd(ℰ_X) = 𝒪_X ⊕ 𝒯_X ⊕ Ω^{[1]}_X ⊕
  \sEnd(𝒯_X).
\]
In particular, we find that the locally free sheaf $\sEnd(ℰ_X)$ locally contains
$𝒯_X$ as a direct summand, which is therefore also locally free.  The positive
solution of the Lipman-Zariski conjecture for klt spaces,
\cite[Thm.~6.1]{GKKP11} or \cite[Thm.~3.8]{Druel13a}, then already implies that
$X$ is smooth.

There is more that we can say.  Assumptions \ref{il:s1} also imply that the
$ℚ$-twisted bundle\footnote{see \cite[Sect.~6.2]{Laz04-II} for the notation used
  here}
\[
  ℰ_X \bigl\langle {\textstyle \frac{1}{n+1}}·[\det ℰ^*_X] \bigr \rangle %
  = ℰ_X \bigl\langle {\textstyle \frac{1}{n+1}}·[K_X]\bigl\rangle
\]
is nef, \cite[Chapt.~IV, Thm.~4.1 and Chap.~II, Def.~6.4]{Nakayama04}, and then
so is the quotient $𝒯_X \bigl\langle \frac{1}{n+1}·[K_X] \bigr\rangle$, see
\cite[Thm.~6.2.12]{Laz04-II}.  Given that $K_X$ is anti-nef, we find that $𝒯_X$
is already nef.  This has consequences: replacing $X$ by a suitable étale cover,
we may assume without loss of generality that the Albanese map
$\alb(X) : X → \Alb(X)$ is a submersion and that its fibres are connected Fano
manifolds whose tangent bundles are likewise nef, \cite[Prop.~3.9 and
Thm.~3.14]{DPS94}.

If $\alb(X)$ has zero-dimensional fibres, then $\alb(X)$ is étale, the variety
$X$ is therefore Abelian, and the proof ends here.  We will therefore assume of
the remainder of the proof that $\alb(X)$ has positive-dimensional fibres, and
let $F ⊆ X$ be such a fibre.  Choose any curve inside $F$ and consider its
normalisation, say $η : C → F$.  Knowing that $F$ is Fano, we find that
$\deg η^*(-K_X) =\deg η^*(-K_F) > 0$, and nefness of
$𝒯_X\langle \frac{1}{n+1}·[K_X]\rangle$ implies that $η^* 𝒯_X$ is already ample,
\cite[Prop.~6.2.11]{Laz04-II}.  As one consequence, we see that $F = X$, for
otherwise there would be a non-trivial surjection
\[
  \begin{tikzcd}
    η^* 𝒯_X \ar[r, two heads] & η^*\sN_{F/X} ≅ 𝒪_{C}^{⊕ \codim_X F},
  \end{tikzcd}
\]
which cannot exist if $η^* 𝒯_X$ is ample.  But then $X$ is a Fano manifold and
hence simply connected, see \cite[Thm.~1.1]{TakayamaSimpleConnectedness} and
references there.  Together with Proposition~\ref{prop:3-3x}, we conclude that
there exists a (Cartier) divisor $D ∈ \Div (X)$ such that
$ℰ_X ≅ 𝒪_X(D)^{⊕ n+1}$.  Taken together with the exact sequence
\eqref{eq:canonicalExtension}, this implies that
\[
  𝒪_X\bigl((n+1)·D\bigr) ≅ \det ℰ_X ≅ 𝒪_X(-K_X).
\]
We find that $c_1(X) = (n+1)·c_1(D)$ and that $D$ is ample.  The classic
Kobayashi-Ochiai Theorem, \cite[Cor.\ on p.~32]{KO73}, then shows that
$X ≅ ℙ^n$.  This concludes the proof of Theorem~\ref{thm:Fano2}.  \qed

%
% Do not edit the following line.  The text is automatically updated by
% subversion.
%
\svnid{$Id: 06-stability.tex 747 2021-04-01 07:12:38Z kebekus $}

\section{Stability criteria for the canonical extension}
\label{ssec:soce}
\approvals{Daniel & yes \\Stefan & yes\\ Thomas & yes}

In order to apply Proposition~\ref{prop:Fano1} we need to find criteria
guaranteeing that the canonical extension $ℰ_X$ of a given variety $X$ is
semistable.

\subsection{Destabilising subsheaves in $ℰ_X$ when $𝒯_X$ is semistable}
\approvals{Daniel & yes \\Stefan & yes\\ Thomas & yes}

It is known that the tangent sheaves of many Fanos are semistable with respect
to the anti-canonical polarisation\footnote{It was in fact conjectured until
  recently that the tangent bundle of every Fano manifold $X$ with Picard number
  $ρ(X)=1$ is stable.  While this conjecture was shown in an impressive number
  of cases, it was recently disproved by Kanemitsu.  We refer the reader to
  Kanemitsu's paper \cite{Kanemitsu20} for details.}, and it seems natural to
ask about semistability of $ℰ_X$ for these varieties.  While even in these cases
one cannot hope that $ℰ_X$ will always be semistable\footnote{See
  \cite[Thm.~3.2]{MR1163733} for an elementary example where $𝒯_X$ is semistable
  while $ℰ_X$ is not.}, we describe potential destabilising subsheaves in $ℰ_X$.

\begin{notation}
  Throughout the present Section~\ref{ssec:soce} we will mostly be concerned
  with Fano manifolds.  To keep notation short, semistability is (unless
  otherwise stated) always understood with respect to the canonical class, and
  we denote the slope with respect to the anticanonical class by $μ(•)$ rather
  than the more correct $μ_{-K_X}(•)$.
\end{notation}

\begin{prop}\label{prop:A}
  Let $X$ be a $ℚ$-Fano variety of dimension $n$.  Suppose that $𝒯_X$ is
  semistable with respect to $-K_X$ and that $𝒮 ⊂ ℰ_X$ is a proper, saturated
  subsheaf with $μ(𝒮) ≥ μ(ℰ_X)$.  Then, the composed map
  $𝒮 ↪ ℰ_X \twoheadrightarrow 𝒯_X$ is injective, and there are inequalities
  \begin{equation}\label{ineq}
    μ(ℰ_X) ≤ μ(𝒮) ≤ μ(𝒯_X).
  \end{equation}
\end{prop}
\begin{proof}
  If the composed map $ψ: 𝒮 → 𝒯_X$ is not injective, then $\ker ψ$ is a subsheaf
  of $𝒪_X$ of rank one and the inequality $μ(𝒮) ≥ μ(ℰ_X)$ yields
  $μ(\Image ψ) > μ(𝒯_X)$, thus contradicting the semistability of $𝒯_X$.  It
  follows that the composed map $ψ$ \emph{is} injective and
  Inequalities~\eqref{ineq} are simply the assumption and the semistability of
  $𝒯_X$.
\end{proof}

We now restrict our attention to Fano manifolds with $ρ(X) = 1$, where the first
Chern classes can and will be considered as numbers; for instance,
$c_1(-K_X) = r$ is the index of $X$.

\begin{prop}\label{prop:B}
  In the setting of Proposition~\ref{prop:A} assume additionally that $X$ is
  smooth with $ρ(X) = 1$, and let $r$ be the index of $X$.  Let $m$ be the rank
  of $𝒮$.  Then, $m < n$ and
  \begin{equation}\label{ineq2}
    \frac{r·m}{n+1} ≤ c_1(𝒮) ≤ \frac{r·m}{n}.
  \end{equation}
\end{prop}
\begin{proof}
  The Inequalities~\eqref{ineq2} are simply reformulations of \eqref{ineq}.  To
  show the bound for $m$, we argue by contradiction and assume that $m = n$.  By
  Proposition~\ref{prop:A}, $𝒮$ is then a reflexive subsheaf of $𝒯_X$.  If
  $𝒮 = 𝒯_X$, then the sequence
  \[
    0 → 𝒪_X → ℰ_X → 𝒯_X → 0
  \]
  is globally split, contradicting Remark~\ref{rem:QFngs}.  It follows that
  $𝒮 ⊊ 𝒯_X$ and that $\det 𝒮 ⊊ 𝒪_X(-K_X)$.  In other words,
  $c_1(𝒮) < c_1(-K_X) = r$.  The left inequality in \eqref{ineq2} will then read
  \[
    \frac{r·n}{n+1} ≤ c_1(𝒮) ≤ r-1,
  \]
  which implies that $r ≥ n+1$.  It follows, for instance by
  \cite[Thm.~1.1]{Kebekus02b}, that $X \simeq ℙ^n$, where $ℰ_X$ is known to be
  stable.  This contradicts the defining property of $𝒮$.
\end{proof}

\begin{cor}\label{cor1}
  Let $X$ be a Fano manifold with $ρ(X) = 1$.  If $𝒮 ⊂ ℰ_X$ is destabilising,
  then $\rank(𝒮) ≥ 2$.
\end{cor}
\begin{proof}
  If $𝒮$ had rank one, then $𝒮$ would be an ample line bundle which at the same
  time is a subsheaf of $𝒯_X$ by Proposition~\ref{prop:A}.  But then $X ≅ ℙ^n$
  by Wahl's theorem, \cite[Thm.~1]{Wah83}, so that again a contradiction has
  been reached.
\end{proof}

\begin{cor}\label{cordim3}
  Let $X$ be a Fano manifold of index one with $ρ(X) = 1$.  Then $ℰ_X$ is
  stable.
\end{cor}
\begin{proof}
  Suppose that $𝒮 ⊂ ℰ_X$ is destabilising.  Recall from \cite[Prop.~2.2]{PW93}
  that $𝒯_X$ is stable.  Proposition~\ref{prop:B} therefore applies to bound the
  integer $c_1(𝒮)$ as follows
  \[
    0 < \frac{1·m}{n+1} ≤ c_1(𝒮) ≤ \frac{1·m}{n} < 1.
  \]
  This is absurd.
\end{proof}

In spite of the elementary results above, the precise relation between
semistability of $𝒯_X$ and $ℰ_X$ is not clear to us.  We would like to pose the
following question.

\begin{problem}
  Let $X$ be a Fano manifold where $𝒯_X$ is (semi)stable with respect to $-K_X$.
  Suppose that $ρ(X) = 1$.  Does this guarantee that the canonical extension
  $ℰ_X$ is (semi)stable as well?
\end{problem}

\subsection{$λ$-stability}
\approvals{Daniel & yes \\Stefan & yes\\ Thomas & yes}

In his work \cite{MR1163733} on Fano manifolds carrying a Kähler-Einstein
metric, Tian introduced the following notion.

\begin{defn}[\protect{$λ$-stability, \cite[Def.~1.3]{MR1163733}}]
  Let $X$ be a normal projective variety and $H$ an ample divisor.  Let $λ > 0$
  be a real number.  We say that a torsion free coherent sheaf $ℰ$ is
  \emph{$λ$-stable with respect to $H$} if the slope inequality
  $μ_H(𝒮) < λ·μ_H(ℰ)$ holds for all saturated subsheaves $𝒮 ⊂ ℰ$ with
  $0 < \rank(𝒮) < \rank(ℰ)$.  Analogously for $λ$-semistable.
\end{defn}

\begin{example}\label{ex:Tian}
  Tian \cite[Thm.~2.1]{MR1163733} has shown that the tangent bundles of Fano
  Kähler-Einstein manifolds of dimension $n$ and with $b_2 = 1$ are
  $\frac{n}{n+1}$-semistable.
\end{example}

An elementary calculation following \cite[Prop.~1.4]{MR1163733} relates
(semi)stability of the canonical extension sheaf to $λ$-(semi)stability of
$𝒯_X$.

\begin{lem}[$λ$-stability criterion]\label{obs:lsc}
  Let $X$ be a $ℚ$-Fano variety of dimension $n$.  Suppose that the tangent
  sheaf $𝒯_X$ is $\frac{n}{n+1}$-(semi)stable.  Then, the canonical extension
  sheaf $ℰ_X$ is (semi)stable.  \qed
\end{lem}

\begin{rem}
  We do not expect the sufficient criterion of Observation~\ref{obs:lsc} to be
  necessary.
\end{rem}

\subsection{Weighted complete intersections}
\approvals{Daniel & yes \\Stefan & yes\\ Thomas & yes}

We will now spell out a few situations where $\frac{n}{n+1}$-stability of $𝒯_X$
can be guaranteed.  Given a Fano manifold $X$ with $b_2(X) = 1$, or equivalently
Picard number $ρ(X) = 1$, we denote by $𝒪_X(1)$ be the ample generator of
$\Pic(X) ≅ ℤ$.  Furthermore, as usual for any coherent sheaf $ℱ$ on $X$ and
$r ∈ ℤ$ we write $ℱ(r) = ℱ ⊗ 𝒪_X(1)^{⊗ r}$, and as above we write
$ω_X ≅ 𝒪_X(-r)$ with $r$ the index of $X$.

Many of the Fano manifolds of interest to us are weighted complete intersections
in weighted projective space, in the following (restrictive) sense, e.g.~see
\cite{MR554521}.

\begin{defn}[Weighted complete intersection]
  A polarised variety $(X, L)$ of dimension $n$ --- i.e., a pair consisting of a
  variety $X$ of dimension $n$ and an ample line bundle $L$ on $X$ --- is called
  a \emph{weighted complete intersection of type $(a_1, …, a_s)$ in
    $ℙ(d_0, d_1, \dots, d_{n+s})$} if the graded algebra
  \[
    R(X,L) = \bigoplus_{k≥ 0} H⁰\bigl(X, \, L^{⊗ k} \bigr)
  \]
  has a system of generators consisting of homogeneous elements
  $ξ_0, ξ_1, \dots, ξ_{n+s}$ with $\deg ξ_j= d_j$ for all $j=0,1, \dots, n+s$
  such that the relation ideal among the $\{ξ_j \}$ is generated by $s$
  homogeneous polynomials $f_1, \dots, f_s$ with $\deg f_l = a_l$ for all
  $l = 1, \dots, s$.

  If $s=1$, we say that $(X,L)$ is a \emph{weighted hypersurface of degree $a_1$
    in $ℙ(d_0, d_1, …, d_{n+1})$}.
\end{defn}

\begin{rem}\label{rem:Fujita=Mori}
  It is noted in \cite[(3.4)]{MR554521} that the above definition is equivalent
  to the one adopted by Mori in \cite[Defn.~3.1]{MR393054}.  In particular, if
  $X$ is smooth, and $X ↪ ℙ(d_0, d_1, \dots, d_{n+r})$ is the natural
  embedding, then $ℙ(d_0, d_1, \dots, d_{n+r})$ is smooth along $X$,
  see \cite[Prop.~1.1]{MR393054}.
\end{rem}

\begin{prop}\label{prop:stab1}
  Let $X$ be a Fano manifold of dimension $n ≥ 3$ with $ρ(X) = 1$ and index
  $r ≤ n$.  Write $δ := n - r$.  Assume that $(X, 𝒪_X(1))$ is a weighted
  complete intersection with
  \begin{equation}\label{cond}
    h⁰ \bigl(X,\, Ω^p_X(p) \bigr) = 0, \quad \text{for all } 1≤ p < n - \frac{δ}{δ +1}(n+1),
  \end{equation}
  Then, $𝒯_X$ is $\frac{n}{n+1}$-semistable.
\end{prop}
\begin{proof}
  We argue as in \cite[Cor.~0.3]{PW93}.  Let $ℱ ⊂ 𝒯_X$ be a saturated and hence
  reflexive subsheaf of rank $1≤ q < n$ and write $\det ℱ ≅ 𝒪_X(k)$.  Then, the
  inequality $μ(ℱ) < μ(𝒯_X) $ to be proven reads
  \begin{equation}\label{IS}
    \frac{k}{q} = μ(ℱ) ≤ \frac{n}{n+1}·μ(𝒯_X) = \frac{n}{n+1}·\frac{r}{n} = \frac{n - δ}{n+1}
  \end{equation}
  We estimate the slope $μ(ℱ)$ by relating it to the existence of twisted forms.
  Taking determinants, we obtain an inclusion $\det ℱ ⊂ Λ^q 𝒯_X$, and hence a
  non-vanishing result
  \begin{equation}\label{eq:C}
    0 \ne h⁰\bigl(X,\, Λ^q 𝒯_X ⊗ 𝒪_X(-k)\bigr) = %
    h⁰ \bigl(X,\, Ω^{n-q}_X(r-k) \bigr).
  \end{equation}
  Using the argument in the proof of Corollary~\ref{cor1} and the
  $\frac{n}{n+1}$-semistability of $𝒯_{ℙ^n}$, we see that we may assume $q ≥ 2$.
  Since $X ≅ \Proj \bigl(R(X, 𝒪_X(1) \bigr)$ is smooth, its affine cone
  $\Spec \bigl(R(X, 𝒪_X(1) \bigr)$ has an isolated singularity at the origin.
  This observation together with Remark~\ref{rem:Fujita=Mori} allows us to apply
  a theorem of Flenner on the non-existence of twisted forms,
  \cite[Satz~8.11]{Fle81},\footnote{See also \cite[Thm.~0.2(d)]{PW93}} and infer
  that $r-k ≥ n-q$.  Equivalently, we found a bound
  \begin{equation}\label{eq:A}
    μ(ℱ) = \frac{k}{q} ≤ \frac{q - δ}{q}.
  \end{equation}
  To make use of \eqref{eq:A}, consider the inequality
  \begin{equation}\label{eq:B}
    (δ+1)·q ≤ δ·(n+1).
  \end{equation}
  There are two cases.  If \eqref{eq:B} holds, then
  \[
    μ(ℱ) \overset{\text{\eqref{eq:A}}}{≤} \frac{q-δ}{q}
    \overset{\text{\eqref{eq:B}}}{≤} \frac{n - δ}{n+1},
  \]
  and Inequality~\eqref{IS} follows.  So, suppose for the remainder of this
  proof that \eqref{eq:B} does \emph{not} hold.  But then
  Assumption~\eqref{cond} implies that
  $H⁰ \bigl(X,\, Ω^{n-q}_X(n-q)\bigr) = \{0\}$.  In view of the non-vanishing
  result \eqref{eq:C}, this implies that $r-k ≥ n-q+1$.  Writing this inequality
  in terms of slopes as in \eqref{eq:A}, and then using that $q < n+1$ by
  definition, we obtain
  \[
    \frac{k}{q} ≤ \frac{q - (δ +1)}{q} < \frac{n - δ}{n+1},
  \]
  establishing \eqref{IS} and completing the proof.
\end{proof}

\begin{rem}\label{rem:extreme_indices}
  Assumption~\eqref{cond} is empty in case $X$ has index one; hence the tangent
  bundle of Fano weighted complete intersections of index one in weighted
  projective spaces is always $\frac{n}{n+1}$-semistable.  At the other end of
  possible indices, if $r=n+1$, then $X ≅ ℙ^n$, whose tangent bundle is
  $\frac{n}{n+1}$-semistable, but not $\frac{n}{n+1}$-stable, as any inclusion
  $𝒪_{ℙ^n}(1) → 𝒯_{ℙ^n}$ shows; if $r = n$, then $X$ is a (hyper)quadric,
  \cite{KO73}, and $T_X$ is $\frac{n}{n+1}$-stable.
\end{rem}

\subsection{Smooth Fano $n$-folds with Picard number one}
\approvals{Daniel & yes \\Stefan & yes\\ Thomas & yes}

In concrete situations, Assumption~\eqref{cond} is often easily verified, as we
will see now.

\begin{prop}\label{stab}
  Let $X$ be a Fano manifold of dimension $n ≥ 3$ with $ρ(X) = 1$.  If $X$ has
  index $r = n-1$, then $𝒯_X$ is $\frac{n}{n+1}$-semistable with respect to
  $-K_X$.
\end{prop}
\begin{proof}
  By assumption, $X$ has index $r = n-1$, so $X$ is a \emph{del Pezzo manifold},
  which are classified, see \cite[Thm.~3.3.1 and Rem.~3.3.2]{MR1668579} or
  \cite[(8.11)]{MR1162108}.  In particular, all such manifolds with $ρ=1$ are
  complete intersections in a weighted projective space or linear sections of
  the Grassmannian $𝔾(2,5)$ embedded into $ℙ⁹$ by the Plücker embedding.
  
  We first treat the case of complete intersections, for which we will verify
  Condition~\eqref{cond}.  The classification exhibits four cases:
  \begin{itemize}
  \item $X$ is a hypersurface of degree $6$ in $ℙ(1, 1, …, 2, 3)$,
    
  \item $X$ is a hypersurface of degree $4$ in $ℙ(1, 1, …, 1, 2)$,
    
  \item $X$ is a cubic in $ℙ^{n+1}$,
    
  \item $X$ is the intersection of two quadrics in $ℙ^{n+2}$.
  \end{itemize}
  
  We treat the first three cases simultaneously and leave the fourth to the
  reader.  Let $ℙ$ denote the relevant (weighted) projective space and let $d$
  be the degree of $X ⊂ ℙ$.  Recall from Remark~\ref{rem:Fujita=Mori} that $ℙ$
  is smooth along $X$, so that we have an exact sequence
  \[
    0 → \sN^*_{X/ℙ} ⊗ Ω_X^{k-1} → Ω^{[k]}_ℙ|_X → Ω^k_X → 0.
  \]
  Tensoring with $𝒪_X(k)$ and taking cohomology, things come down to showing
  that
  \begin{equation}\label{eq:W1}
    H⁰ \bigl(X,\, Ω^{[k]}_ℙ(k)|_X\bigr) =
    H¹ \bigl(X,\, Ω^{k-1}_X(k-d)\bigr) = \{0\} \quad\quad \text{ for all } 1≤ k ≤ ⌊ \frac{n-1}{2}⌋.
  \end{equation}
  The right-hand side of \eqref{eq:W1} is settled by
  \cite[Satz~8.11(1.b)]{Fle81}, whereas the left-hand side follows from
  \[
    H⁰ \bigl(ℙ,\, Ω_{ℙ}^{[k]}(k) \bigr) = \{0\} %
    \quad\text{and}\quad %
    H¹ \bigl(ℙ,\, Ω^{[k]}_{ℙ}(k-d) \bigr) = \{0\}\quad \quad \text{for all } 1 ≤ k ≤
    ⌊\frac{n-1}{2}⌋.
  \]
  Both are of course classical, see, e.g., \cite[Thm.\ 2.3.2 and Cor.~2.3.4;
  notation on p.~47]{Dol82}.
  
  Thus, it remains to treat linear sections of the Grassmannian $𝔾 := 𝔾(2,5)$
  embedded by Plücker in $ℙ⁹$ (in dimensions $3$ up to $6$).  If $\dim X = 6$,
  then $X = 𝔾(2,5)$ is Kähler-Einstein, so we are done by Example~\ref{ex:Tian}.
  Therefore, we are reduced to $3 ≤ n = \dim X ≤ 5$.  Suppose $𝒯_X$ is not
  $\frac{n}{n+1}$-semistable and consider a $\frac{n}{n+1}$-destabilising
  subsheaf $ℱ ⊂ 𝒯_X$.  Since $𝒯_X$ is stable, \cite[Thm.~2.3]{PW93}, and has
  slope $\frac{r}{n}$, elementary slope considerations give $ \dim X ≥ 4$,
  $\rank(ℱ) = n-1$, and $c_1(ℱ) = n-2$.  Similar to \eqref{eq:C}, taking
  determinants hence yields
  \begin{equation}\label{eq:nonvanishing}
    H⁰ \bigl(X,\, Ω¹_X(1)\bigr) \ne \{0\}.
  \end{equation}
  
  We will see that this is impossible.  First, suppose that $n = 5$ and look at
  the structure sheaf sequence for the hyperplane section $X ⊂ 𝔾 ⊂ ℙ⁹$, twisted
  with $Ω¹_𝔾(1)$.  Since
  \[
    h⁰\bigl(𝔾,\, Ω¹_𝔾(1)\bigr) = h¹\bigl(𝔾,\, Ω¹_𝔾(1)\bigr) %
    = h²\bigl(𝔾,\, Ω¹_𝔾\bigr)= 0
  \]
  by \cite[Lem.~0.1]{PW93}, and since $h¹\bigl(𝔾,\, Ω¹_𝔾\bigr) = b_2(𝔾) = 1$, we
  see that $h⁰\bigl(X,\, Ω¹_𝔾(1)|_X\bigr) = 1$ and
  $h¹\bigl(X, Ω¹_𝔾(1)|_X \bigr) = 0$.  Therefore, taking cohomology of the
  twisted conormal sequence
  \[
    0 → 𝒪_X → Ω¹_𝔾(1)|_X → Ω¹_X(1) → 0
  \]
  we obtain
  \begin{equation}\label{eq:ngleich5}
    H⁰\bigl(X,\, Ω¹_X(1)\bigr) = H¹\bigl(X,\, Ω¹_X(1)\bigr)= \{0\},
  \end{equation}
  contradicting \eqref{eq:nonvanishing} and proving the assertion for $n=5$.

  Second, look at a further hyperplane section $Y ⊂ X$, i.e, the case $n=4$, and
  repeat the above argument: use the non-vanishing \eqref{eq:nonvanishing}, but
  this time for $Y$, together with
  $H⁰\bigl(X,\, Ω¹_X(1)\bigr) = H¹\bigl(X,\, Ω¹_X(1)\bigr) = \{0\}$, see
  \eqref{eq:ngleich5}, and $h¹\bigl(X,\, Ω¹_X\bigr) = b_2(X) = 1$ to get
  $h⁰\bigl(Y,\, Ω¹_X(1)|_Y\bigr) = 1$, which then contradicts the vanishing
  obtained from the conormal sequence for $Y$ in $X$.
\end{proof}

\begin{cor}\label{cor:Fano3folds}
  Let $X$ be a smooth Fano threefold.  If $ρ(X) = 1$, then $𝒯_X$ is
  $\frac{n}{n+1}$-semistable, and the canonical extension $ℰ_X$ is semistable.
\end{cor}
\begin{proof}
  Fano threefolds of Picard number one and index $r = 1, 3,4$ are taken care of
  by Remark~\ref{rem:extreme_indices}, whereas those of index $r = 2$ are
  covered by Proposition~\ref{stab}.  Lemma~\ref{obs:lsc} implies the statement
  about the canonical extension.
\end{proof}

\begin{rem}
  Using the classification of Fano manifolds of coindex 3 and the same methods,
  Theorem \ref{stab} certainly holds in case $r = \dim X -2$ as well.
\end{rem}

\subsection{Relation to K-stability}
\approvals{Daniel & yes \\Stefan & yes\\ Thomas & yes}

On the one hand, many of the examples considered in the preceding section can be
shown to be K-polystable using (in each single case) technically much more
advanced methods; see the overview provided in
\cite[Sect.~6.12]{Xu-K-stabilityNotes}.  Then, the positive solution of the
YTD-conjecture shows that these admit a Kähler-Einstein metric, which in turn
implies that $𝒯_X$ is $\frac{n}{n+1}$-semistable by Tian's original result, see
Example~\ref{ex:Tian}.  Note that according to \emph{loc.~cit.}, it is an open
question whether all smooth Fano threefolds of Picard number one are
K-semistable.

On the other hand, if $X$ is a linear section of $𝔾(2,5)$ of dimension $4$ or
$5$, then the pair $(X, -K_X)$ is not K-semistable, see \cite{MR3715691}.
Together with Proposition~\ref{stab} and Lemma~\ref{obs:lsc}, this shows that
$ℰ_X$ might be semistable (making Proposition~\ref{prop:Fano1} and
Theorem~\ref{thm:Fano2} applicable) even for K-unstable Fano manifolds $X$,
which in particular do not carry Kähler-Einstein metrics.

Finally, we note that \cite[Thm.~1.5]{MR3731324} together with the degree
estimate \cite[(4.11)]{MR1163733} implies that for any Fano manifold $X$ such
that $(X, -K_X)$ is K-semistable the canonical extension $ℰ_X$ is semistable
with respect to $-K_X$, as seems to be well-known to experts.\footnote{We thank
  Henri Guenancia for confirming our corresponding educated guess.}

%
% Do not edit the following line.  The text is automatically updated by
% subversion.
%
\svnid{$Id: 07-examples.tex 498 2020-06-09 13:38:11Z peternell $}

\section{Examples}\label{sect:examples}

\subsection{Necessity of stability assumptions}
\label{ssec:nosta}
\approvals{Daniel & yes \\Stefan & yes\\ Thomas & yes}

The following elementary examples showing that certain stability conditions are
necessary for the $ℚ$-Miyaoka-Yau inequality and the characterisation of the
extreme case to hold are well-known to experts.  We recall them for the
convenience of the reader.

\begin{example}[Projectivised bundle]
  Consider the 4-dimensional Fano manifold
  $X := ℙ_{ℙ³} \bigl( 𝒪_{ℙ³} ⊕ 𝒪_{ℙ³}(3) \bigr)$.  Then $[-K_X]⁴ = 800$, whereas
  $[-K_X]² · c_2(X) = 296$.  It follows that $X$ violates the Bogomolov-Gieseker
  Inequality~\eqref{eq:BG} as well as the Miyaoka-Yau Inequality~\eqref{eq:MY}.
  In particular, it follows from \cite[Thm.~3.4.1]{MR2665168} that $𝒯_X$ is not
  semistable with respect to $-K_X$.
\end{example}

\begin{example}[Weighted projective space]\label{ex:wps}
  The three-dimensional weighted projective space $X := ℙ(1,1,1,3)$ has one
  isolated canonical Gorenstein singularity of type $\frac{1}{3}(1,1,1)$, which
  admits a crepant resolution.  Moreover, $X$ is Fano and its anticanonical
  volume is $[-K_X]³ = 72$; see \cite{285155} and \cite[Ex.~1.4]{MR2141325}.
  Clearly, $χ(𝒪_X) = 1$, so that \cite[Cor.~10.3]{Reid87} implies\footnote{Note
    that the ``basket'' is empty due to the existence of a crepant resolution.}
  that
  \[
    [-K_X]·\what{c}_2\left( Ω^{[1]}_X \right) = [-K_X]·c_2(X) = 24.
  \]
  Hence, $X$ realises equality in the $ℚ$-Bogomolov-Gieseker
  Inequality~\eqref{eq:BG} and violates the $ℚ$-Miyaoka-Yau
  Inequality~\eqref{eq:MY}.
\end{example}

\begin{rem}[Explanation of Example~\ref{ex:wps}]
  Prokhorov \cite[Thm~1.5]{MR2141325} proved that any Fano threefold $X$ with
  only canonical Gorenstein singularities satisfies the bound $(-K_X)³ ≤ 72$,
  and in case of equality either $X ≅ ℙ(1,1,1,3)$ or $X ≅ ℙ(1,1,4,6)$.  With
  this in mind, consider $Y := ℙ_{ℙ²} \bigl( 𝒪_{ℙ²} ⊕ 𝒪_{ℙ²}(3) \bigr)$ and
  $σ: Y → Z$ be the blow down of the exceptional section; then $Z$ is a
  three-dimensional Fano Gorenstein variety of index $2$ with a canonical
  singularity and $[-K_Y]³ = 72$; consequently, $Z ≅ ℙ(1,1,1,3)$.  Using this
  realisation, it is easy to check that $𝒯_{ℙ(1,1,1,3)}$ is not semistable.  In
  fact, $σ_*(𝒯_{Y/ℙ²})$ is a torsion free destabilising subsheaf.
\end{rem}

If $X$ is a smooth Fano threefold, the above phenomena do not occur, even
without assuming $𝒯_X$ to be semistable.  In fact, classification shows that
\[
  [-K_X]³ ≤ 4³ = 64,
\]
for any smooth Fano threefold, \cite{MR641971, MR715648, MR1668579, MR1969009}.
Since
\[
  -K_X · c_2(X) = 24·χ(X,𝒪_X) = 24,
\]
Inequality~\eqref{eq:MY} holds.  Moreover, equality in \eqref{eq:MY} occurs if
and only if $X = ℙ³$.  This remains true for threefolds with at worst Gorenstein
terminal singularities, since these admit smoothings by \cite{MR1489117}.  We
refer the reader to \cite{MR2141325} for a further discussion.

\subsection{Equality in the Miyaoka-Yau Inequality~\eqref{eq:MY}}
\label{sec:7-6}
\approvals{Daniel & yes \\Stefan & yes\\ Thomas & yes}

Examples of Fano varieties realising the bound \eqref{eq:MY} are produced by the
following result.

\begin{prop}\label{prop:examples}
  Let $V$ be an $n$-dimensional complex vector space, $n ≥ 2$.  Let $G < \GL(V)$
  be a finite subgroup having the following two properties.
  \begin{enumerate}
  \item No element of $G ∖ \{e\}$ is a homothety.
    
  \item No element of $G$ is a quasi-reflection.\footnote{An element
      $g ∈ \GL(V)$ is called \emph{quasi-reflection} if $1$ is an eigenvalue of
      $g$ with geometric multiplicity equal to $\dim V -1$.}
  \end{enumerate}
  Let $W$ be the direct sum of $V$ with a $1$-dimensional trivial
  $G$-representation, $W := ℂ ⊕ V$.  Then, the quotient map $γ: ℙ(W) → ℙ(W)/G$
  for the induced $G$-action on $ℙ(W)$ is quasi-étale with Galois group $G$.
\end{prop}
\begin{proof}
  Introduce linear coordinates $z_1, …, z_n$ on $V$, $z_0$ on $ℂ$, and the
  corresponding homogeneous coordinates $[z_0: z_1: … : z_n]$ on $ℙ(W)$.  By
  construction, the point $[1:0:…:0] ∈ ℙ(W)$ is fixed by $G$.  Moreover, the
  $G$-invariant (and hence $γ$-saturated) neighbourhood $U := \{z_0 ≠ 0\}$ is
  $G$-equivariantly isomorphic to $V$ in such a way that $[1:0:…:0]$ is mapped
  to $0 ∈ V$.  In particular, $G$ acts effectively on $ℙ(W)$.  As $G$ does not
  have any quasi-reflections, the restriction of $γ$ to $U$ is quasi-étale.
 
  It therefore remains to exclude ramification of $γ$ along $Z := ℙ(W) ∖ U$,
  which is $G$-equivariantly isomorphic to $ℙ(V)$, where $G$ act on the latter
  space via its inclusion into $\GL(V)$.  Assume that there is an element
  $g ∈ G$ that fixes $Z$ pointwise.  Then, $g$ acts on $V$ via a homothety, so
  that $g=\{e\}$ by assumption.  Therefore, $γ$ is unramified at the generic
  point of $Z$ and therefore quasi-étale, as claimed.
\end{proof}

\begin{example}
  Let $ξ$ be a non-trivial third root of unity, and let $G = ℤ/3ℤ$ act on
  $V = ℂ²$ by
  \[
    [m]·(z_1, z_2) = (ξ^{m}· z_1, ξ^{2m}· z_2).
  \]
  The action is faithful, so that we may consider $G$ as a subgroup of
  $\GL(2, ℂ)$.  Moreover, the assumptions of Proposition~\ref{prop:examples} are
  fulfilled, so that $ℙ(ℂ⊕V) / G$ is a $ℚ$-Fano surface realising equality in
  the $ℚ$-Miyaoka-Yau Inequality~\eqref{eq:MY}.
\end{example}

\begin{example}
  Let $G$ be a non-Abelian finite simple group, and let $ρ: G → \GL(V)$ be any
  non-trivial finite-dimensional representation; $\dim V =:n$ is necessarily
  greater than two.  As $\ker(ρ)$ is a normal subgroup of $G$, the
  representation is faithful, so that $G < \GL(V)$.  Similarly, as $G$ is
  non-Abelian, $Z(G)$ has to be trivial; in particular, $G$ does not contain
  homotheties.  Moreover, the subgroup of $G$ generated by quasi-reflections is
  normal, \cite[proof of Prop.~6]{MR210944}, so likewise trivial, since the
  Shephard-Todd classification of complex reflection groups,
  \cite[Chap.~8]{MR2542964}, does not contain non-Abelian simple groups.
  Consequently, $G$ fulfils the assumptions of Proposition~\ref{prop:examples}
  above, and $ℙ(ℂ⊕V) / G$ is a $ℚ$-Fano $n$-fold realising equality in the
  $ℚ$-Miyaoka-Yau Inequality~\eqref{eq:MY}.
 
  As subexamples, we may take $G=A_5$ and $ρ$ the three-dimensional (real)
  representation realising $A_5$ as the rotational symmetry group of the
  icosahedron, or $G$ the monster group with its smallest non-trivial
  representation, which has dimension $196.883$.
\end{example}

We find that every finite group appears as the Galois group for one of our
examples.

\begin{cor}\label{cor:7-7}
  Let $G$ be any finite group.  Then, there exists a finite-dimensional,
  faithful complex representation $ρ: G ↪ \GL(W)$ such that the quotient map
  $γ: ℙ(W) → ℙ(W)/G$ for the induced $G$-action on $ℙ(W)$ is quasi-étale with
  Galois group $G$.
\end{cor}
\begin{proof}
  Let $V_0$ be a faithful complex representation of $G$.  Then, the direct sum
  $V:= V_0 ⊕ ℂ$ of $V_0$ with the trivial representation is also faithful, and
  no element of $G ∖ \{e\}$ acts on $V$ by a homothety.
 
  We now follow the argument given in \cite[proof of
  Cor.~5]{BraunKLTFiniteness}.  If there is no element of $G$ acting via a
  quasi-reflection on $V$, then by Proposition~\ref{prop:examples} above we may
  take $W= V ⊕ ℂ$, the direct sum of $V$ with a further trivial representation,
  on which $G$ also acts without non-trivial homotheties.  Otherwise, we
  consider
  \begin{equation}\label{eq:direct_sum_with_itself}
    V' := V ⊕ V.
  \end{equation}
  Note that no element of $G ∖ \{e\}$ acts on $V'$ by a homothety.  Suppose that
  there is a $g ∈ G$ acting as a quasi-reflection on $V'$.  Consider $V'$ as a
  representation of the subgroup $\langle g \rangle ⊂ G$ generated by $g$; by
  assumption, this contains a pointwise fixed hyperplane $H$.  As the
  representation $V'$ is obviously reducible, a non-trivial decomposition into
  subrepresentations being given by \eqref{eq:direct_sum_with_itself}, it
  follows from the first paragraph of the proof of \cite[Theorem]{MR2912485}
  that one of the copies of $V$ has to be contained in $H$.  In other words,
  $\langle g \rangle$, and in particular $g$, acts trivially on $V$,
  contradicting faithfulness of $V$.  We may then apply
  Proposition~\ref{prop:examples} again in order to conclude.
\end{proof}

Corollary~\ref{cor:7-7} is particularly interesting because of the following
topological observation, which shows that every finite group appears as the
fundamental group of the smooth locus in a $ℚ$-Fano variety that realises
equality in the $ℚ$-Miyaoka-Yau Inequality~\eqref{eq:MY}.

\begin{lem}[Topological properties of $ℙ^n/G$]
  Let $X ≅ ℙ^n/G$ for some finite subgroup $G < \Aut_𝒪(ℙ^n)$ such that the
  quotient map $γ: ℙ^n → X$ is quasi-étale.  Then, $X$ is simply connected and
  $π_1(X_{\reg}) = G$.
\end{lem}
\begin{proof}
  Simple connectedness of $X$ follows from the fact that every $g ∈ G$ has a
  fixed point in the simply connected space $ℙ^n$ together with \cite[main
  result, first sentences of intro]{MR187244}.  To show that
  $π_1(X_{\reg}) = G$, note that $γ^{-1}(X_{\reg}) ⊂ ℙ^n$ has a complement of
  codimension at least two by assumption, and is therefore simply connected.
  Moreover, $G$ acts freely on it with quotient $X_{\reg}$, so that
  $π_1(X_{\reg}) = G$, as claimed.
\end{proof}


\begin{thebibliography}{GKKP11}

\bibitem[Arm65]{MR187244}
Mark~A. Armstrong.
\newblock On the fundamental group of an orbit space.
\newblock {\em Proc. Cambridge Philos. Soc.}, 61:639--646, 1965.
\newblock
  \href{https://doi.org/10.1017/s0305004100038974}{DOI:10.1017/s0305004100038974}.

\bibitem[Ati57]{Ati57}
Michael~F. Atiyah.
\newblock Complex analytic connections in fibre bundles.
\newblock {\em Trans. Amer. Math. Soc.}, 85:181--207, 1957.
\newblock \href{https://doi.org/10.2307/1992969}{DOI:10.2307/1992969}.

\bibitem[Bis09]{MR2545454}
Indranil Biswas.
\newblock Semistability and restrictions of tangent bundle to curves.
\newblock {\em Geom. Dedicata}, 142:37--46, 2009.
\newblock
  \href{https://doi.org/10.1007/s10711-009-9356-3}{DOI:10.1007/s10711-009-9356-3}.
  Preprint \href{https://arxiv.org/abs/0901.4161}{arXiv:0901.4161}.

\bibitem[Bra20]{BraunKLTFiniteness}
Lukas Braun.
\newblock The local fundamental group of a {K}awamata log terminal singularity
  is finite.
\newblock Preprint \href{https://arxiv.org/abs/2004.00522}{arXiv:2004.00522}.,
  2020.

\bibitem[CO75]{CH75}
Bang-Yen Chen and Koichi Ogiue.
\newblock Some characterizations of complex space forms in terms of {C}hern
  classes.
\newblock {\em Quart. J. Math. Oxford Ser. (2)}, 26(104):459--464, 1975.
\newblock
  \href{https://doi.org/10.1093/qmath/26.1.459}{DOI:10.1093/qmath/26.1.459}.

\bibitem[DGP20]{DGP20}
Stéphane Druel, Henri Guenancia, and Mihai Păun.
\newblock A decomposition theorem for $\mathbb{Q}$-{F}ano {K}ähler-{E}instein
  varieties.
\newblock Preprint \href{http://arxiv.org/abs/2008.05352}{arXiv:2008.05352},
  August 2020.

\bibitem[Dol82]{Dol82}
Igor Dolgachev.
\newblock Weighted projective varieties.
\newblock In {\em Group actions and vector fields ({V}ancouver, {B}.{C}.,
  1981)}, volume 956 of {\em Lecture Notes in Math.}, pages 34--71. Springer,
  Berlin, 1982.
\newblock \href{https://doi.org/10.1007/BFb0101508}{DOI:10.1007/BFb0101508}.

\bibitem[Don02]{MR1959581}
Simon~K. Donaldson.
\newblock Holomorphic discs and the complex {M}onge-{A}mpère equation.
\newblock {\em J. Symplectic Geom.}, 1(2):171--196, 2002.
\newblock
  \href{https://projecteuclid.org/euclid.jsg/1092316649}{euclid.jsg/1092316649}.

\bibitem[DPS94]{DPS94}
Jean-Pierre Demailly, Thomas Peternell, and Michael Schneider.
\newblock Compact complex manifolds with numerically effective tangent bundles.
\newblock {\em J. Algebraic Geom.}, 3(2):295--345, 1994.
\newblock Available from the author's web site
  \href{https://www-fourier.ujf-grenoble.fr/~demailly/manuscripts/dps1.pdf}{https://www-fourier.ujf-grenoble.fr/$\sim$demailly/manuscripts/dps1.pdf}.

\bibitem[Dru14]{Druel13a}
Stéphane Druel.
\newblock The {Z}ariski-{L}ipman conjecture for log canonical spaces.
\newblock {\em Bull. London Math. Soc.}, 46(4):827--835, 2014.
\newblock \href{https://doi.org/10.1112/blms/bdu040}{DOI:10.1112/blms/bdu040}.
  Preprint \href{http://arxiv.org/abs/1301.5910}{arXiv:1301.5910}.

\bibitem[Eis95]{E95}
David Eisenbud.
\newblock {\em Commutative algebra with a view toward algebraic geometry},
  volume 150 of {\em Graduate Texts in Mathematics}.
\newblock Springer-Verlag, New York, 1995.
\newblock
  \href{https://doi.org/10.1007/978-1-4612-5350-1}{DOI:10.1007/978-1-4612-5350-1}.

\bibitem[FL81]{FL81}
William Fulton and Robert Lazarsfeld.
\newblock Connectivity and its applications in algebraic geometry.
\newblock In {\em Algebraic geometry ({C}hicago, {I}ll., 1980)}, volume 862 of
  {\em Lecture Notes in Math.}, pages 26--92. Springer, Berlin, 1981.
\newblock \href{https://doi.org/10.1007/BFb0090889}{DOI:10.1007/BFb0090889}.

\bibitem[Fle81]{Fle81}
Hubert Flenner.
\newblock Divisorenklassengruppen quasihomogener {S}ingularit\"{a}ten.
\newblock {\em J. Reine Angew. Math.}, 328:128--160, 1981.
\newblock
  \href{https://doi.org/10.1515/crll.1981.328.128}{DOI:10.1515/crll.1981.328.128}.

\bibitem[Fle84]{Flenner84}
Hubert Flenner.
\newblock Restrictions of semistable bundles on projective varieties.
\newblock {\em Comment. Math. Helv.}, 59(4):635--650, 1984.
\newblock \href{https://doi.org/10.1007/BF02566370}{DOI:10.1007/BF02566370}.

\bibitem[Fuj80]{MR554521}
Takao Fujita.
\newblock On the hyperplane section principle of {L}efschetz.
\newblock {\em J. Math. Soc. Japan}, 32(1):153--169, 1980.
\newblock
  \href{https://doi.org/10.2969/jmsj/03210153}{DOI:10.2969/jmsj/03210153}.

\bibitem[Fuj90]{MR1162108}
Takao Fujita.
\newblock {\em Classification theories of polarized varieties}, volume 155 of
  {\em London Mathematical Society Lecture Note Series}.
\newblock Cambridge University Press, Cambridge, 1990.
\newblock
  \href{https://doi.org/10.1017/CBO9780511662638}{DOI:10.1017/CBO9780511662638}.

\bibitem[Fuj17]{MR3715691}
Kento Fujita.
\newblock Examples of {K}-unstable {F}ano manifolds with the {P}icard number 1.
\newblock {\em Proc. Edinb. Math. Soc. (2)}, 60(4):881--891, 2017.
\newblock
  \href{https://doi.org/10.1017/S0013091516000432}{DOI:10.1017/S0013091516000432}.
  Preprint \href{https://arxiv.org/abs/1508.04290}{arXiv:1508.04290}.

\bibitem[GKKP11]{GKKP11}
Daniel Greb, Stefan Kebekus, Sándor~J. Kovács, and Thomas Peternell.
\newblock Differential forms on log canonical spaces.
\newblock {\em Inst. {H}autes {É}tudes Sci.~{P}ubl.~{M}ath.}, 114(1):87--169,
  November 2011.
\newblock
  \href{https://doi.org/10.1007/s10240-011-0036-0}{DOI:10.1007/s10240-011-0036-0}.
  An extended version with additional graphics is available as
  \href{http://arxiv.org/abs/1003.2913}{arXiv:1003.2913}.

\bibitem[GKP16a]{GKP15}
Daniel Greb, Stefan Kebekus, and Thomas Peternell.
\newblock Movable curves and semistable sheaves.
\newblock {\em Int. Math. Res. Not.}, 2016(2):536--570, 2016.
\newblock \href{https://doi.org/10.1093/imrn/rnv126}{DOI:10.1093/imrn/rnv126}.
  Preprint \href{http://arxiv.org/abs/1408.4308}{arXiv:1408.4308}.

\bibitem[GKP16b]{GKP13}
Daniel Greb, Stefan Kebekus, and Thomas Peternell.
\newblock Étale fundamental groups of {K}awamata log terminal spaces, flat
  sheaves, and quotients of abelian varieties.
\newblock {\em Duke Math. J.}, 165(10):1965--2004, 2016.
\newblock
  \href{https://doi.org/10.1215/00127094-3450859}{DOI:10.1215/00127094-3450859}.
  Preprint \href{http://arxiv.org/abs/1307.5718}{arXiv:1307.5718}.

\bibitem[GKP21]{GKP20b}
Daniel Greb, Stefan Kebekus, and Thomas Peternell.
\newblock Projectively flat klt varieties.
\newblock {\em J. Éc. polytech. Math.}, 8:1005--1036, 2021.
\newblock \href{https://doi.org/10.5802/jep.164}{DOI:10.5802/jep.164}. Preprint
  \href{https://arxiv.org/abs/2010.06878}{arXiv:2010.06878}.

\bibitem[GKPT19]{GKPT19b}
Daniel Greb, Stefan Kebekus, Thomas Peternell, and Behrouz Taji.
\newblock The {M}iyaoka-{Y}au inequality and uniformisation of canonical
  models.
\newblock {\em Ann. Sci. Éc. Norm. Supér. (4)}, 52(6):1487--1535, 2019.
\newblock \href{https://doi.org/10.24033/asens.2414}{DOI:10.24033/asens.2414}.
  Preprint \href{https://arxiv.org/abs/1511.08822}{arXiv:1511.08822}.

\bibitem[Gro66]{EGA4-3}
Alexandre Grothendieck.
\newblock Éléments de géométrie algébrique. {IV}. Étude locale des
  schémas et des morphismes de schémas {III}.
\newblock {\em Inst. Hautes Études Sci. Publ. Math.}, (28):255, 1966.
\newblock Revised in collaboration with Jean Dieudonné.
  \href{http://www.numdam.org/item?id=PMIHES_1966__28__5_0}{numdam.PMIHES-1966-28-5-0}.

\bibitem[Gro70]{MR0262386}
Alexander Grothendieck.
\newblock Représentations linéaires et compactification profinie des groupes
  discrets.
\newblock {\em Manuscripta Math.}, 2:375--396, 1970.
\newblock \href{https://doi.org/10.1007/BF01719593}{DOI:10.1007/BF01719593}.

\bibitem[GW20]{GrebWong}
Daniel Greb and Michael~Lennox Wong.
\newblock Canonical complex extensions of {K}ähler manifolds.
\newblock {\em J. Lond. Math. Soc. (2)}, 2020.
\newblock \href{https://doi.org/10.1112/jlms.12287}{DOI: 10.1112/jlms.12287}.
  Preprint \href{https://arxiv.org/abs/}{arXiv:1807.01223}.

\bibitem[Har77]{Ha77}
Robin Hartshorne.
\newblock {\em Algebraic geometry}.
\newblock Springer-Verlag, New York, 1977.
\newblock Graduate Texts in Mathematics, No. 52.
  \href{https://doi.org/10.1007/978-1-4757-3849-0}{DOI:10.1007/978-1-4757-3849-0}.

\bibitem[HL10]{MR2665168}
Daniel Huybrechts and Manfred Lehn.
\newblock {\em The geometry of moduli spaces of sheaves}.
\newblock Cambridge Mathematical Library. Cambridge University Press,
  Cambridge, second edition, 2010.
\newblock
  \href{https://doi.org/10.1017/CBO9780511711985}{DOI:10.1017/CBO9780511711985}.

\bibitem[IP99]{MR1668579}
Vasily~A. Iskovskih and Yuri~G. Prokhorov.
\newblock Fano varieties.
\newblock In {\em Algebraic geometry, {V}}, volume~47 of {\em Encyclopaedia
  Math. Sci.}, pages 1--247. Springer, Berlin, 1999.

\bibitem[Kan20]{Kanemitsu20}
Akihiro Kanemitsu.
\newblock {\em J. Reine Angew. Math.}, 2020.
\newblock Ahead of print.
  \href{https://doi.org/10.1515/crelle-2020-0043}{DOI:10.1515/crelle-2020-0043}.
  Preprint \href{https://arxiv.org/abs/1912.12617}{arXiv:1912.12617}.

\bibitem[Keb02]{Kebekus02b}
Stefan Kebekus.
\newblock Characterizing the projective space after {C}ho, {M}iyaoka and
  {S}hepherd-{B}arron.
\newblock In {\em Complex geometry (Göttingen, 2000)}, pages 147--155.
  Springer, Berlin, 2002.
\newblock
  \href{https://doi.org/10.1007/978-3-642-56202-0_10}{DOI:10.1007/978-3-642-56202-0\_10}.

\bibitem[KM98]{KM98}
János Kollár and Shigefumi Mori.
\newblock {\em Birational geometry of algebraic varieties}, volume 134 of {\em
  Cambridge Tracts in Mathematics}.
\newblock Cambridge University Press, Cambridge, 1998.
\newblock With the collaboration of C.\ H.\ Clemens and A.\ Corti, Translated
  from the 1998 Japanese original.
  \href{https://doi.org/10.1017/CBO9780511662560}{DOI:10.1017/CBO9780511662560}.

\bibitem[KO73]{KO73}
Shoshichi Kobayashi and Takushiro Ochiai.
\newblock Characterizations of complex projective spaces and hyperquadrics.
\newblock {\em J. Math. Kyoto Univ.}, 13:31--47, 1973.
\newblock
  \href{https://doi.org/10.1215/kjm/1250523432}{DOI:10.1215/kjm/1250523432}.

\bibitem[Kob87]{Kob87}
Shoshichi Kobayashi.
\newblock {\em Differential geometry of complex vector bundles}, volume~15 of
  {\em Publications of the Mathematical Society of Japan}.
\newblock Iwanami Shoten and Princeton University Press, Princeton, NJ, 1987.
\newblock Kanô Memorial Lectures, 5.

\bibitem[Laz04]{Laz04-II}
Robert Lazarsfeld.
\newblock {\em Positivity in algebraic geometry. {II}}, volume~49 of {\em
  Ergebnisse der Mathematik und ihrer Grenzgebiete. 3. Folge. A Series of
  Modern Surveys in Mathematics [Results in Mathematics and Related Areas. 3rd
  Series. A Series of Modern Surveys in Mathematics]}.
\newblock Springer-Verlag, Berlin, 2004.
\newblock Positivity for vector bundles, and multiplier ideals.
  \href{https://doi.org/10.1007/978-3-642-18808-4}{DOI:10.1007/978-3-642-18808-4}.

\bibitem[Li17]{MR3731324}
Chi Li.
\newblock Yau-{T}ian-{D}onaldson correspondence for {K}-semistable {F}ano
  manifolds.
\newblock {\em J. Reine Angew. Math.}, 733:55--85, 2017.
\newblock
  \href{https://doi.org/10.1515/crelle-2014-0156}{DOI:10.1515/crelle-2014-0156}.
  Preprint \href{https://arxiv.org/abs/1302.6681}{arXiv:math/1302.6681}.

\bibitem[LT09]{MR2542964}
Gustav~I. Lehrer and Donald~E. Taylor.
\newblock {\em Unitary reflection groups}, volume~20 of {\em Australian
  Mathematical Society Lecture Series}.
\newblock Cambridge University Press, Cambridge, 2009.

\bibitem[LT18]{LT18}
Steven Lu and Behrouz Taji.
\newblock A characterization of finite quotients of abelian varieties.
\newblock {\em Int. Math. Res. Not. IMRN}, (1):292--319, 2018.
\newblock \href{https://doi.org/10.1093/imrn/rnw251}{DOI:10.1093/imrn/rnw251}.
  Preprint \href{http://arxiv.org/abs/1410.0063}{arXiv:1410.0063}.

\bibitem[MM83]{MR715648}
Shigefumi Mori and Shigeru Mukai.
\newblock On {F}ano {$3$}-folds with {$B_{2}\geq 2$}.
\newblock In {\em Algebraic varieties and analytic varieties ({T}okyo, 1981)},
  volume~1 of {\em Adv. Stud. Pure Math.}, pages 101--129. North-Holland,
  Amsterdam, 1983.
\newblock
  \href{https://doi.org/10.2969/aspm/00110101}{DOI:10.2969/aspm/00110101}.

\bibitem[MM03]{MR1969009}
Shigefumi Mori and Shigeru Mukai.
\newblock Erratum: ``{C}lassification of {F}ano 3-folds with {$B_2\geq 2$}''
  [{M}anuscripta {M}ath. {\bf 36} (1981/82), no. 2, 147--162; {MR}0641971
  (83f:14032)].
\newblock {\em Manuscripta Math.}, 110(3):407, 2003.
\newblock
  \href{https://doi.org/10.1007/s00229-002-0336-2}{DOI:10.1007/s00229-002-0336-2}.

\bibitem[MM82]{MR641971}
Shigefumi Mori and Shigeru Mukai.
\newblock Classification of {F}ano {$3$}-folds with {$B_{2}\geq 2$}.
\newblock {\em Manuscripta Math.}, 36(2):147--162, 1981/82.
\newblock \href{https://doi.org/10.1007/BF01170131}{DOI:10.1007/BF01170131}.

\bibitem[Mor75]{MR393054}
Shigefumi Mori.
\newblock On a generalization of complete intersections.
\newblock {\em J. Math. Kyoto Univ.}, 15(3):619--646, 1975.
\newblock
  \href{https://doi.org/10.1215/kjm/1250523007}{DOI:10.1215/kjm/1250523007}.

\bibitem[Nak04]{Nakayama04}
Noboru Nakayama.
\newblock {\em Zariski-decomposition and abundance}, volume~14 of {\em MSJ
  Memoirs}.
\newblock Mathematical Society of Japan, Tokyo, 2004.
\newblock
  \href{https://doi.org/10.2969/msjmemoirs/014010000}{DOI:10.2969/msjmemoirs/014010000}.

\bibitem[Nam97]{MR1489117}
Yoshinori Namikawa.
\newblock Smoothing {F}ano {$3$}-folds.
\newblock {\em J. Algebraic Geom.}, 6(2):307--324, 1997.

\bibitem[Oda13]{MR3010808}
Yuji Odaka.
\newblock The {GIT} stability of polarized varieties via discrepancy.
\newblock {\em Ann. of Math. (2)}, 177(2):645--661, 2013.
\newblock
  \href{https://doi.org/10.4007/annals.2013.177.2.6}{DOI:10.4007/annals.2013.177.2.6}.
  Preprint \href{https://arxiv.org/abs/0807.1716}{arXiv:0807.1716}.

\bibitem[Pri67]{MR210944}
David Prill.
\newblock Local classification of quotients of complex manifolds by
  discontinuous groups.
\newblock {\em Duke Math. J.}, 34:375--386, 1967.
\newblock
  \href{https://doi.org/10.1215/S0012-7094-67-03441-2}{DOI:10.1215/S0012-7094-67-03441-2}.

\bibitem[Pro05]{MR2141325}
Yuri~G. Prokhorov.
\newblock The degree of {F}ano threefolds with canonical {G}orenstein
  singularities.
\newblock {\em Mat. Sb.}, 196(1):81--122, 2005.
\newblock
  \href{https://doi.org/10.1070/SM2005v196n01ABEH000873}{DOI:10.1070/SM2005v196n01ABEH000873}.

\bibitem[PW95]{PW93}
Thomas Peternell and Jarosław~A. Wiśniewski.
\newblock On stability of tangent bundles of {F}ano manifolds with {$b_2=1$}.
\newblock {\em J. Algebraic Geom.}, 4(2):363--384, 1995.

\bibitem[Rei87]{Reid87}
Miles Reid.
\newblock Young person's guide to canonical singularities.
\newblock In {\em Algebraic geometry, Bowdoin, 1985 (Brunswick, Maine, 1985)},
  volume~46 of {\em Proc. Sympos. Pure Math.}, pages 345--414. Amer. Math.
  Soc., Providence, RI, 1987.

\bibitem[Sem92]{MR1165352}
Stephen Semmes.
\newblock Complex {M}onge-{A}mpère and symplectic manifolds.
\newblock {\em Amer. J. Math.}, 114(3):495--550, 1992.
\newblock \href{https://doi.org/10.2307/2374768}{DOI:10.2307/2374768}.

\bibitem[Sta17]{285155}
Jason Starr.
\newblock Volume of {$-K_X$} for a weighted projective variety.
\newblock MathOverflow, 2017.
\newblock \href{https://mathoverflow.net/q/285155}{mathoverflow.net/q/285155
  (version: 2017-11-03)}.

\bibitem[Tak00]{TakayamaSimpleConnectedness}
Shigeharu Takayama.
\newblock Simple connectedness of weak {F}ano varieties.
\newblock {\em J. Algebraic Geom.}, 9(2):403--407, 2000.

\bibitem[Tia92]{MR1163733}
Gang Tian.
\newblock On stability of the tangent bundles of {F}ano varieties.
\newblock {\em Internat. J. Math.}, 3(3):401--413, 1992.
\newblock
  \href{https://doi.org/10.1142/S0129167X92000175}{DOI:10.1142/S0129167X92000175}.

\bibitem[Tia96]{MR1603624}
Gang Tian.
\newblock Kähler-{E}instein metrics on algebraic manifolds.
\newblock In {\em Transcendental methods in algebraic geometry ({C}etraro,
  1994)}, volume 1646 of {\em Lecture Notes in Math.}, pages 143--185.
  Springer, Berlin, 1996.
\newblock \href{https://doi.org/10.1007/BFb0094304}{DOI:10.1007/BFb0094304}.

\bibitem[Wah83]{Wah83}
Jonathan Wahl.
\newblock {A cohomological characterization of $\mathbb P_n$}.
\newblock {\em Invent. Math.}, 72(4):315--322, 1983.
\newblock \href{https://doi.org/10.1007/BF01389326}{DOI:10.1007/BF01389326}.

\bibitem[Weh73]{MR0335656}
Bertram A.~F. Wehrfritz.
\newblock {\em Infinite linear groups. {A}n account of the group-theoretic
  properties of infinite groups of matrices}.
\newblock Springer-Verlag, New York, 1973.
\newblock Ergebnisse der Mathematik und ihrer Grenzgebiete, Band 76.
  \href{https://doi.org/10.1007/978-3-642-87081-1}{DOI:10.1007/978-3-642-87081-1}.

\bibitem[Xu20]{Xu-K-stabilityNotes}
Chenyang Xu.
\newblock K-stability of {F}ano varieties: an algebro-geometric approach.
\newblock Preprint
  \href{https://arxiv.org/pdf/2011.10477.pdf}{arXiv:2011.10477}, November 2020.

\bibitem[YJ12]{MR2912485}
JianMing Yu and GuangFeng Jiang.
\newblock Reducibility of finite reflection groups.
\newblock {\em Sci. China Math.}, 55(5):947--948, 2012.
\newblock
  \href{https://doi.org/10.1007/s11425-011-4341-3}{DOI:10.1007/s11425-011-4341-3}.

\end{thebibliography}
\end{document}